\documentclass[12pt]{article}
\usepackage{tikz}
\usetikzlibrary{arrows}
\usepackage{tkz-graph}
\usepackage[leqno]{amsmath}
\usepackage{amssymb,amsthm,url,nicefrac,bm,ulem}
\tikzstyle{vertex}=[circle, draw, inner sep=0pt, minimum size=6pt]
\numberwithin{equation}{section}

\usepackage[letterpaper]{geometry}
\date{}

\theoremstyle{plain}
\newtheorem{thm}{Theorem}
\newtheorem{cor}{Corollary}
\newtheorem{lem}{Lemma}

\theoremstyle{definition}
\newtheorem*{example}{Example}
\newtheorem*{defn}{Definition}
\newtheorem*{ack}{Acknowledgments}

\theoremstyle{remark}

\newtheorem*{remark}{Remark}
\newtheorem*{remarks}{Remarks}

\usepackage[sort]{natbib}
\usepackage{varioref}
\labelformat{section}{Section~#1}
\labelformat{figure}{Figure~#1}
\labelformat{table}{Table~#1}
\labelformat{thm}{Theorem~#1}
\labelformat{cor}{Corollary~#1}

\usepackage[font=small,labelfont=bf]{caption}

\DeclareMathOperator\pf{PF}
\DeclareMathOperator\sh{Sh}

\newcommand\for{\qquad\text{for }}
\newcommand\pfn{\pf_n}
\newcommand\calf{\mathcal{F}}
\newcommand\calfn{\calf_n}
\newcommand\tcalfn{\tilde\calf_n}
\newcommand\hy{\hat{y}}
\newcommand\ha{\hat{a}}
\newcommand\tilhlam{\tilde{H}_\lambda}
\newcommand\barp{\bar{P}}

\begin{document}

\title{Probabilizing Parking Functions}
\author{Persi Diaconis\\\textit{Departments of Mathematics and Statistics}\\\textit{Stanford University}%
\and Angela Hicks\\\textit{Department of Mathematics}\\\textit{Lehigh University}}
\maketitle

\begin{abstract}
We explore the link between combinatorics and probability generated by the question ``What does a random parking function look like?'' This gives rise to novel probabilistic interpretations of some elegant, known generating functions. It leads to new combinatorics: how many parking functions begin with $i$? We classify features (e.g., the full descent pattern) of parking functions that have exactly the same distribution among parking functions as among all functions. Finally, we develop the link between parking functions and Brownian excursion theory to give examples where the two ensembles differ.
\end{abstract}

\section{Introduction}\label{sec1}

Parking functions are a basic combinatorial object with applications in combinatorics, group theory, and computer science. This paper explores them by asking the question, ``What does a typical parking function look like?'' Start with $n$ parking spaces arranged in a line ordered left to right, as:
\begin{equation*}\begin{array}{ccccc}
-&-&-&\dots&-\\[-6pt]
1&2&3& &n
\end{array}\end{equation*}
There are $n$ cars, each having a preferred spot. Car $i$ wants $\pi_i$, $1\leq i\leq n$, with $1\leq\pi_i\leq n$. The first car parks in spot $\pi_1$, then the second car tries to park in spot $\pi_2$: if this is occupied, it takes the first available spot to the right of $\pi_2$. This continues; if at any stage $i$ there is no available spot at $\pi_i$ or further right, the process aborts.

\begin{defn}
A parking function is a sequence $\pi=(\pi_1,\pi_2,\dots,\pi_n)$ with $1\leq\pi_i\leq n$ so that all cars can park. Let $\pfn$ be the set of parking functions.
\end{defn}

\begin{example}
$(1,1,\dots,1)$ is a parking function. A sequence with two $n$'s is not. When $n=3$, there are 16 parking functions,
\begin{equation*}
111,112,121,211,113,131,311,122,212,221,123,132,213,231,312,321.
\end{equation*}
\end{example}

Using the pigeonhole principle, it is easy to see that a parking function must have at least one $\pi_i=1$, it must have at least two values $\leq 2$, and so on; $\pi$ is a parking function if and only if
\begin{equation}
\#\{k:\pi_k\leq i\}\geq i,\qquad1\leq i\leq n.
\label{11}
\end{equation}
Dividing through by $n$ (both inside and outside) this becomes
\begin{equation}
F^\pi(x)\geq x\for x=\frac{i}{n},\ 1\leq i\leq n,
\label{12}
\end{equation}
with $F^\pi(x)=\nicefrac1{n}\cdot(\#\{k:\pi_k\leq nx\})$ the distribution function of the measure
\begin{equation*}
\frac1{n}\sum_{i=1}^n\delta_{\pi_i/n}.
\end{equation*}

Following a classical tradition in probabilistic combinatorics we ask the following questions for the distribution of features of a randomly chosen $\pi\in\pfn$:
\begin{itemize}
\item What is the chance that $\pi_i=j$?
\item What is the distribution of the descent pattern in $\pi$?
\item What is the distribution of the sum $\sum_{i=1}^n\pi_i$?
\end{itemize}


\noindent It is easy to generate a parking function randomly on the computer. \ref{fig:first_cars} shows a histogram of the values of $\pi_1$ based on $50,000$ random choices for $n=100$.\begin{figure}\begin{center}\includegraphics[width=9cm]{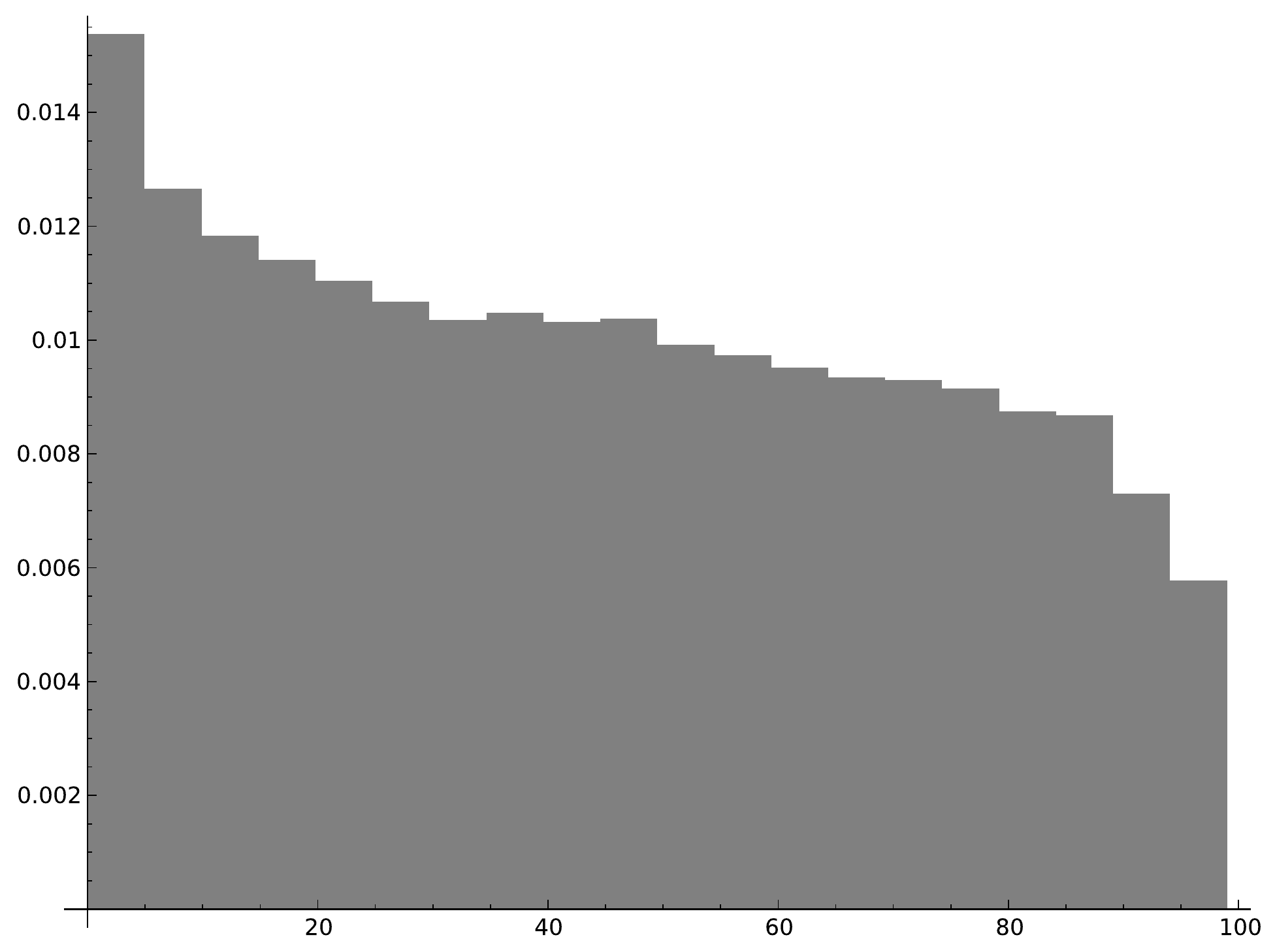}\end{center}\caption{This histogram gives the distribution of values of $\pi_1$ (the first preference) in 50,000 parking functions of size 100 chosen uniformly at random.}\label{fig:first_cars}
\end{figure} There seems to be a smooth curve. How does it behave? Notice that because of invariance under permutations (see \ref{11}), the answer gives the distribution of any coordinate $\pi_i$. \ref{sec2} gives a closed formula as well as a simple large-$n$ approximation.

In \ref{sec3} we mine the treasure trove of elegant generating functions derived by the combinatorics community, leaning on \citet{yan}. These translate into classical binomial, Poisson, and central limit theorems for features like
\begin{equation}\begin{aligned}
&\bullet\text{the number of $i$ with $\pi_i=1$};\\
&\bullet\text{the number of repeats $\pi_i=\pi_{i+1}$};\\
&\bullet\text{the number of lucky cars, where car $i$ gets spot $\pi_i$}.
\end{aligned}\label{13}
\end{equation}

One unifying theme throughout may be called the \textit{equivalence of ensembles}. Let $\calfn=\{f:[n]\to[n]\}$ (where $[n]$ is the standard shorthand for $\{1,\cdots,n\}$) so $|\calfn|=(n)^n$ . Of course $\pfn\subseteq\calfn$ and $$|\pfn|/|\calfn|=(1+\nicefrac1{n})^{n-1}/n\sim\frac{e}{(n)}.$$
A random parking function is a random function in $\calfn$ conditioned on being in $\pfn$. In analogy with equivalence of ensembles in statistical mechanics, it is natural to expect that for some features, the distribution of the features in the ``micro-canonical ensemble'' ($\pfn$) should be close to the features in the ``canonical ensemble'' ($\calfn$). As we shall see, there is a lot of truth in this heuristic--it holds for the features in \eqref{13} among others. Further, its failures are interesting.

\ref{sec4} develops classes of features where the heuristic is exact. This includes the full descent theory, the up/down pattern in $\pi$. It explains some of the formulas behind features \eqref{13} and gives new formulas.

\ref{sec5} uses some more esoteric probability. For a random $\pi$, it is shown that $[F^\pi(x)-x]_{0\leq x\leq 1}$ converges to a Brownian excursion process. This allows us to determine the large-$n$ behavior of
\begin{equation}\begin{aligned}
&\bullet\#\{k:\pi_k\leq i\};\\
&\bullet\max_{1\leq k\leq n}\{k:\pi_k\leq i\}-i;\\
&\bullet\sum_{i=1}^n\pi_i.
\end{aligned}\label{14}
\end{equation}
These distributions differ from their counterparts for a random $f\in\calfn$. For example, $(\sum_{i=1}^nf_i-\nicefrac{n^2}2)/n^{3/2}$ has a limiting normal distribution but $\sum_{i=1}^n(\pi_i-\nicefrac{n^2}2)/n^{3/2}$ has a limiting Airy distribution.

\ref{sec6} tries to flesh out the connection between parking functions and Macdonald polynomials; it may be read now for further motivation. \ref{sec7} briefly discusses open problems (e.g., cycle structure) and generalizations.

\paragraph{Literature review}
Parking functions were introduced by \citet{konheim} to study the widely used storage device of hashing. They showed there are $(n+1)^{n-1}$ parking functions, e.g., 16 when $n=3$. Since then, parking functions have appeared all over combinatorics; in the enumerative theory of trees and forests \citep{chassaing}, in the analysis of set partitions \citep{MR1444167}, hyperplane arrangements \citep{MR1627378} \citep{MR835214}, polytopes \citep{MR1902680} \citep{MR2576844}, chip firing games \citep{MR1756151}, and elsewhere.

\citet{yan} gives an accessible, extensive survey which may be supplemented by \citet{stanley99}, \citet{MR3383893}, and \citet{MR3281144}. Our initial interest in parking functions comes through their relation to the amazing Macdonald polynomials \citep{haiman02} \citep{MR3443860}. 

\paragraph{Generating a random parking function} As previously mentioned, it is quite simple to generate a parking function uniformly at random. To select $\pi \in \pfn$:
\begin{enumerate}
\item Pick an element $\pi\in(\mathbb{Z}/(n+1)\mathbb{Z})^n$, where here (as in later steps) we take the equivalence class representatives $1,\cdots,(n+1)$.
\item If $\pi\in \pfn$ (i.e.\ if $\pi'=\operatorname{sort}(\pi)$ is such that $\pi'_i\leq i$ for all $i$), return $\pi$.
\item Otherwise, let $\pi:=\pi+ (1,\cdots,1)$ (working mod $n+1$). Return to (2).
\end{enumerate}
In fact, for every $\pi\in(\mathbb{Z}/(n+1)\mathbb{Z})^n$, there is exactly one $k\in\mathbb{Z}/(n+1)\mathbb{Z} $ such that $\pi+k(1,\cdots,1)$ is a parking function.  This process is suggested by Pollak's original proof of the number of parking functions as given in \citet{foata1974mappings}.

\section{Coordinates of Random Parking Functions}\label{sec2}

Let $\pfn$ be the set of parking functions $\pi=(\pi_1,\pi_2,\dots,\pi_n)$. This section gives exact and limiting approximations for the distribution of $\pi_1$ and, by symmetry, $\pi_i$, when $\pi\in\pfn$ is chosen uniformly. For comparison, consider $\calfn=\{f:[n]\to[n]\}$. Then $\pfn\subseteq\calfn$, $|P\calfn|/|\calfn|=(1+\nicefrac1{n})^{n-1}/n$, and a natural heuristic compares a random $\pi$ with a random $f$. Section \ref{sec21} looks at a single coordinate, while Section \ref{sec22} looks at the joint distribution of $k$ coordinates.
\subsection{Single Coordinates}\label{sec21}
For random $f$,
\begin{equation*}
P(f_1=j)=\frac1{n},\quad E(f_1)\sim\frac{n}2.
\end{equation*}
The results below show that for $j$ fixed and $n$ large,
\begin{equation*}
P(\pi_1=j)\sim\frac{1+Q(j)}{n},\quad P(\pi_1=n-j)\sim\frac{1-Q(j+2)}{n},\quad E(\pi_1)\sim\frac{n}2,
\end{equation*}
with $Q(j)=P(X\geq j)$, $P(X=j)=e^{-j}j^{j-1}/j!$, the Borel distribution on $j=1,2,3,\dots$. The extremes are $P(\pi_1=1)\sim\nicefrac2{n}$, $P(\pi_1=n)\sim\nicefrac1{en}$, with $P(\pi_1=j)$ monotone decreasing in $j$. Since $Q(j)\sim\sqrt{\nicefrac2{(\pi j)}}$ when $j$ is large, the bulk of the values are close to uniform. Figure 1 in the introduction shows that $n$ must be large for this ``flatness'' to take hold.

The arguments depend on a combinatorial description of the parking functions that begin with $k$. Let
\begin{equation*}
A_{\pi_2,\dots,\pi_n}=\left\{j:(j,\pi_2,\pi_3,\dots,\pi_n)\in\pfn\right\}.
\end{equation*}
If $k\in A_{\pi_2,\dots,\pi_n}$ then $k_1\in A_{\pi_2,\dots,\pi_n}$ for all $1\leq k_1\leq k$ so $A_{\pi_2,\dots,\pi_n}=[k]$ for some fixed value of $k$ (or empty) and we need only determine the set's maximal element. 
\begin{defn}[Parking Function Shuffle] Say that $\pi_2,\dots,\pi_n$ is a \textit{parking function shuffle} of $\alpha\in\pf_{k-1}$ and $\beta\in\pf_{n-k}$ (write $\pi_2,\dots,\pi_n\in\sh(k-1,n-k)$) if $\pi_2,\cdots,\pi_n$ is a shuffle of the two words $\alpha$ and $\beta+(k-1,\dots,k-1)$.
\end{defn}
\begin{example}
With $n=7$ and $k=4$, $(2,\emph{4},1,\emph{5},\emph{7},2,\emph{6})$ is a shuffle of $(2,1,2)$ and $(1,2,4,3)$. The main result is
\end{example}

\begin{thm}
$A_{\pi_2,\dots,\pi_n}=[k]$ if and only if $(\pi_2,\dots,\pi_n)\in\sh(k-1,n-k)$.
\label{thm1}
\end{thm}
\begin{cor}
The number of $\pi\in\pfn$ with $\pi_1=k$ is
\begin{equation*}
\sum_{s=0}^{n-k}\binom{n-1}{s}(s+1)^{s-1}(n-s)^{n-s-2}.
\end{equation*}
\label{cor1}
\end{cor}
\noindent Note that this quantity decreases as $k$ increases as we have fewer resulting summands.

\begin{cor} Let $P(X=j)=e^{-j}\frac{j^{j-1}}{j!}$, so that $X$ has a Borel distribution. Then for any parking function $\pi$ uniformly chosen of size $n$, fixed $j$ and $k$ and $0<j\leq k< n$, $$P\left(\pi_1=j\text{ and }A_{\pi_2,\dots,\pi_n}=[k]\right)\sim \frac{1}{n}P(X=k).$$
If $k$ is close to $n$, by contrast, and $0<j<n-k$
$$P\left(\pi_1=j\text{ and }A_{\pi_2,\dots,\pi_n}=[k]\right)\sim \frac{1}{n}P(X=n-k+1).$$ 
\label{cor15}
\end{cor}
The explicit formulae and asymptotics derived above allow several simple probabilistic
interpretations. Let $\pi$ be a parking function This determines
$\pi_2,...,\pi_n$ and so define $K_\pi$ such that $[K_\pi] = A_{\pi_2,\cdots,\pi_n}$, that is, $K_\pi$ is the largest
possible first element consistent with the last $n-1$. This $K_\pi$ is
a random variable and we may ask about it's distribution and about the
conditional distribution of $\pi_1$ given that $K_\pi = k$. the
following argument shows that with probability tending to 1, $K_\pi$
is close to $n$ AND the conditional distribution is uniform on $1,2,\cdots, K_\pi$.

\begin{cor} Let $q(j) = e^{-j} j^{j-1}/ j!$ be the Borel distribution on
$1,2,\cdots$. For fixed $j$, as $n$ tends to infinity, $P(K_\pi = n-j )$ tends
to $q(j)$. Further, the conditional distribution of $\pi_1$ given $K_\pi
= n-j$ is close to the uniform distribution on 1$,\cdots,n-j$ in the Levy
metric.
\label{cor17}
\end{cor}
\begin{cor}\label{cor2}
For $\pi$ uniformly chosen in $\pfn$, $j\geq 1$ fixed, and $n$ large relative to $j$,
\begin{equation*}
P(\pi_1=n-j)\sim\frac{P(X\leq j+1)}{n}\qquad\text{with }\ P(X=j)=e^{-j}\frac{j^{j-1}}{j!}.
\end{equation*}
In particular, $P(\pi_1=n)\sim\nicefrac1{(en)}$.

\end{cor}

\begin{cor}
For $\pi$ uniformly chosen in $\pfn$, $j\geq1$ fixed, and $n$ large,
\begin{equation*}
P(\pi_1=j)\sim\frac{1+P(X\geq j)}{n}.
\end{equation*}
In particular, $P(\pi_1=1)\sim\nicefrac2{n}$.
\label{cor3}
\end{cor}

Corollaries 3 and 4 show that ``in the corners'' the distribution of $\pi_1$ and $f_1$ are different. However for most $j$, $P(\pi_1=j)\sim P(f(1)=j)=\nicefrac1{n}$. Indeed, this is true in a fairly strong sense. Define the total variation distance between two measures $P$ and $\barp$ on $\{1,2,\dots,n\}$ to be
\begin{equation}
\left\|P-\barp\right\|:=\max_{A\subseteq[n]}\left|P(A)-\barp(A)\right|=\frac12\sum_{j=1}^n\left|P(j)-\barp(j)\right|=\frac12\max_{\|f\|_{\infty\leq i}}\left|P(f)-\barp(f)\right|.
\label{21}
\end{equation}
The first equality above is a definition; the others are easy consequences.

\begin{cor}
Let $P_n(j)$ and $\barp_n(j)$ be the probabilities associated with $\pi_1$, $f_1$, from uniformly chosen $\pi\in\pfn$, $f\in\calfn$. Then,
\begin{equation*}
\left\|P_n-\barp_n\right\|\stackrel{{n}}{\underset{\infty}{\longrightarrow}}0.
\end{equation*}
\label{cor4}
\end{cor}

\begin{remark}
With more work, a rate of convergence should follow in Corollary 6. Preliminary calculations suggest $\|P_n-\barp_n\|\leq\nicefrac{c}{\sqrt{n}}$ for $c$ universal. We conjecture that, in variation distance, the joint distribution of $\pi_1,\dots,\pi_k$ is close to the joint distribution of $f_1,\dots,f_k$ ($k$ fixed, $n$ large).
\end{remark}

The final result gives an asymptotic expression for the mean $E(\pi_1)=\sum_{j=1}^njP(\pi_1=j)$. The correction term is of interest in connection with the area function of \ref{sec5}.

\begin{thm}
For $\pi$ uniformly chosen in $\pfn$,
\begin{equation*}
E(\pi_1)=\frac{n}2-\frac{\sqrt{2\pi}}{4}n^{1/2}\left(1+o(1)\right).
\end{equation*}
Note that $E(\pi_1)\sim E(f_1)$.
\label{thm2}
\end{thm}

The proof of \ref{thm1} is broken into four short, easy lemmas. These are given next, followed by proofs of the corollaries and \ref{thm2}. The results make nice use of Abel's extension of the binomial theorem.

\begin{lem}
If $(\pi_2,\dots,\pi_n)\in\sh(k-1,n-k)$ then $\pi=(k,\pi_2,\dots,\pi_n)$ is a parking function.
\label{lem1}
\end{lem}

\begin{proof}
Certainly, $\#\{i:\pi_i\leq j\}\geq j$ for $j<k$. Thus
\begin{equation*}
\#\{i:\pi_i\leq k\}=\#\{i:\pi_i\leq k-1\}+1\geq k.
\end{equation*}
Finally, for $j>k$,
\begin{align*}
\#\{i:\pi_i\leq j\}&=\#\{i:\pi_i\leq k\}+\#\{i:k<\pi_i\leq j\}\\
&\geq k+\#\{i:0<\pi_i-k\leq j-k\}\geq k+j-k=j,
\end{align*}
where the last inequality comes from the fact that the cars greater than $k$ come from a parking function $\beta$.
\end{proof}

\begin{lem}
If $(\pi_2,\dots,\pi_n)\in\sh(k-1,n-k)$ then $\pi=(k+1,\pi_2,\dots,\pi_n)$ is not a parking function.
\label{lem2}
\end{lem}

\begin{proof}
Assume not. Then $\#\{i:\pi_i>k\}=n-k+1$ and thus $\#\{i:\pi_i\leq k\}\leq k-1$.
\end{proof}

\begin{lem}
If $\pi=(k,\pi_2,\dots,\pi_n)$ is a parking function, some subsequence of $\pi$ must be a parking function of length $k-1$.
\label{lem3}
\end{lem}

\begin{proof}
Take a subsequence formed by the (first) $k-1$ cars with value at most $k-1$.
\end{proof}

\begin{lem}
If $\pi=(k,\pi_2,\dots,\pi_n)$ is a parking function but $\pi'=(k+1,\pi_2,\dots,\pi_n)$ is not, some subsequence of $\pi$ is of the form $\beta+(k,\dots,k)$ where $\beta\in\pf_{n-k}$.
\label{lem4}
\end{lem}

\begin{proof}
It must be that $\pi$ has exactly $k$ cars less than or equal to $k$ (including the first car). Thus call $\beta'$ the subsequence with exactly $n-k$ cars of value greater than $k$ and let $\beta=\beta'-(k,\dots,k)$. Since
\begin{equation*}
\#\{i:k\leq\pi_i\leq k+j\}\geq j,
\end{equation*}
$\beta'$ is a parking function as desired.
\end{proof}

\begin{proof}[Proof of \ref{cor1}]
Note that, if $\pi_1=j$, then $A_{\pi_2,\dots,\pi_n}=[a]$ for some $a\geq j$. Thus the number of parking functions with $\pi_1=j$ is
\begin{equation*}
\sum_{a=j}^n\binom{n-1}{a-1}a^{a-2}(n-a+1)^{n-a-1}=\sum_{s=0}^{n-j}\binom{n-1}{s}(s+1)^{s-1}(n-s)^{n-s-2}.
\end{equation*}
Here $\binom{n-1}{a-1}$ accounts for the positions of the smaller parking functions in the shuffle, $a^{a-2}$ is the number of smaller parking functions, and $(n-a+1)^{n-a-1}$ is the number of larger parking functions in a possible shuffle. The equality follows by changing $a$ to $n-s$.
\end{proof}
\begin{proof}[Proof of \ref{cor15}] 
For fixed $k$ and large $n$, we have 
\begin{align*}
\frac{1}{(n+1)^{n-1}}\binom{n-1}{k-1}k^{k-2}(n-k+1)^{n-k-1}&\sim \frac{(n-1)^{k-1}k^{k-2}(n-k+1)^{n-k-1}}{(k-1)!(n+1)^{n-1}}\\
&\sim \frac{k^{k-2}(n-k+1)^{n-k-1}}{(k-1)!(n+1)^{n-k}}\\
&\sim \frac{k^{k-1}e^{-k}}{k!(n-k+1)}\\
&=\frac{1}{n-k+1}P(X=k)
\end{align*}
The proof for $k$ close to $n$ follows similarly.
\end{proof}
\begin{proof}[Proof of \ref{cor17}] From the formulas the chance that $K_\pi = k$ is $$k {n-1
\choose k-1} k^{k-2} (n-k +1)^{n-k-1} / (n+1)^{n-1}.$$ For $k$ of form $n-j$,
the asymptotic used above shows that this tends to $q(j)$. For any fixed
$\pi_1$ from $1,2,\cdots,n-j$ the chance of $\pi_1$ and $K_\pi$ is asymptotic
to $q(j)/(n-j)$ by a similar calculation ($j$ fixed and $n$ tending to
infinity). The result follows. 
\end{proof}
\begin{proof}[Proof of \ref{cor2}]
When $j=0$, from \ref{cor1},
\begin{equation*}
P(\pi_1=n)=\frac{n^{n-2}}{(n+1)^{n-1}}=\frac1{n}\frac1{(1+\nicefrac1{n})^{n-1}}\sim\frac1{en}=\frac{P(X\leq 2 )}{n}.
\end{equation*}
The proof for large n and fixed $j$ follows similarly.
\end{proof}

\begin{proof}[Proof of Corollary 5]
Abel's generalization of the binomial theorem \citep{riordan,pitman} gives
\begin{equation}
\sum_a\binom{n}{a}(x+a)^{a-1}(y+n-a)^{n-a-1}=(x^{-1}+y^{-1})(x+y+n)^{n-1}.
\label{22}
\end{equation}
Apply this to \ref{cor1} when $n\rightarrow n-1$, $x\rightarrow 1$ and $y\rightarrow 1$, and to see that the number of parking functions with $\pi_1=1$ is $2(n+1)^{n-2}$. Now standard asymptotics gives
\begin{equation*}
P\{\pi_1=1\}=\frac2{n+1}\sim\frac{1+P(X\geq2))}{n}.
\end{equation*}
The proof for large $n$ and any fixed $j$ follows similarly.
\end{proof}

\begin{proof}[Proof of \ref{cor4}]
From \ref{cor1}, $P(\pi_1=j)$ is monotone decreasing in $j$. Now \ref{cor2} and \ref{cor3} show that, for large $n$,
\begin{equation*}
P(\pi_1=j)=\frac{1+P(X\geq j+1)}{n}(1+o(1)),\quad P(\pi_1=n-j)=\frac{P(X\leq j+1))}{n}(1+o(1)),
\end{equation*}
with $Q_n$ tending to zero. Thus
\begin{equation*}
\|P_n-Q_n\|=\frac12\sum_{j=1}^n|P(\pi_1=j)-P(f_1=j)|=\frac12\sum_{j=1}^n|nP(\pi=j)-1|\frac1{n}.
\end{equation*}
The terms in absolute value are bounded and are uniformly small for $j^*\leq j\leq n-j^*$ for appropriately chosen $j^*$. Any bounded number of terms when multiplied by $\nicefrac1{n}$ tend to zero. The result follows.
\end{proof}

\begin{proof}[Proof of \ref{thm2}]
From \ref{cor1},
\begin{align*}
\sum_{t=1}^nt\#\{\pi\in\pfn:\pi_1=t\}&=\sum_{t=1}^n\sum_{s=0}^{n-t}\binom{n-1}{s}(s+1)^{s-1}(n-s)^{n-s-2}\\
&=\sum_{s=0}^{n-1}\binom{n-1}{s}(n-s)^{n-s-2}(s+1)^{s-1}\sum_{t=1}^{n-s}1\\
&=\frac12\sum_{s=0}^{n-1}\binom{n-1}{s}(n-s)^{n-s-1}(s+1)^{s-1}(n-s+1)\\
&=\frac12\sum_{s=0}^{n-1}\binom{n-1}{s}(n-s)^{n-s-1}(s+1)^{s-1}\\
&\qquad\qquad\qquad+\frac12\sum_{s=0}^n\binom{n-1}{s}(n-s)^{n-s}(s+1)^{s-1}\\
&=\frac12(I+II).
\end{align*}
Use Abel's identity above to see
\begin{equation*}
I=(n+1)^{n-1}.
\end{equation*}
A variant of Abel's identity gives
\begin{align*}
II&=\sum_{k=0}^{n-1}\binom{n-1}{k}(n+1)^k(n-k-1)!(n-k)\\
&=(n-1)!\sum_{k=0}^{n-1}\frac{(n+1)^k}{k!}(n-k)\\
&=(n-1)!\left[n\sum_{k=0}^{n-1}\frac{(n+1)^k}{k!}-\sum_{k=1}^{n-1}\frac{(n+1)^k}{(k-1)!}\right]\\
&=n(n+1)^{n-1}-(n-1)!\sum_{k=0}^{n-2}\frac{(n+1)^{k}}{k!}.
\end{align*}
Combining terms and dividing by $(n+1)^{n-1}$,
\begin{equation*}
E(\pi_1)=\frac12+\frac{n}2-\frac{(n-1)!}{2(n+1)^{n-1}}\sum_{k=0}^{n-2}\frac{(n+1)^k}{k!}.
\end{equation*}
From Stirling's formula,
\begin{equation*}
\frac{(n-1)!}{(n+1)^{n-1}}\sim\sqrt{2\pi}n^{1/2}e^{-(n+1)}.
\end{equation*}
Finally,
\begin{equation*}
e^{-(n+1)}\sum_{k=0}^{n-2}\frac{(n+1)^k}{k!}
\end{equation*}
equals the probability that a Poisson random variable with parameter $n+1$ is less than or equal to $n-2$. This is asymptotic to $\nicefrac12$ by the law of large numbers.
\end{proof}
\begin{remarks}
\begin{itemize}
\item \citet{eu2005enumeration} also enumerate parking functions by leading terms (in a more general case) using a bijection to labeled rooted trees. Knuth gives a generating function for the area of a parking function which he justifies by a nice argument using the final value of the parking function \citep{knuth1998linear}.  Not surprisingly, since any rearrangement of a parking function is also a parking function, he also relies on a shuffle of parking functions in the sense described above to get essentially the same formula in a slightly different language.  \citet{foata1974mappings} give the same generating function and several others based on the values of the entries of $\pi$.
\item We find it somewhat mysterious that the Borel distribution sums
to one: $$\sum_{j=1}^{\infty} \frac{e^{-j} j^{j-1}}{j!} =1.$$ The usual proof of
this interprets the summands as the probability that a classical
Galton-Watson process with Poisson(1) births dies out with $j$ total
progeny. The argument makes nice use of the Lagrange inversion
theorem. Our colleague Kannan Soundararajan produced an elementary proof
of a more general fact, which while new to us, Richard Stanley kindly pointed out, also occurs in \citet[p.28]{stanleyenumerative}. For $0 \leq x \leq 1$, $$\sum_{j=1}^\infty
\frac{e^{-xj} (xj)^{j-1} }{j!} =1.$$ Indeed, for $k\geq 1$, the coefficient of $x^k$ on
the left hand side is $$\sum_{j=1} ^{k+1} \frac{j^{j-1} (-j)^{k-j+1}}{j!(k-j+1)}
= (-1)^k+1 \sum_{ j=1}^{k+1} (-1)^j j^k {k+1 \choose j}.$$ This last is
zero for all $k \geq 1$ since $$\sum_{j = 0}^n(-1)^j j^k {n \choose j} = 0$$ for all fixed $k$ from 0 to
$n-1$ as one sees by applying the transformation $T(f(x)) = xf'(x)$ to the
function $(1-x)^n$ and setting $x = 1$.  
\end{itemize}
\end{remarks}
\subsection{Many coordinates of a random parking function are jointly uniform.}\label{sec22} \ref{cor4} shows that any single coordinate is close to uniform in a strong sense (total variation distance). In this section we show that any $k$ coordinates are close to uniform, where $k$ is allowed to grow with $n$ as long as $k\ll\sqrt{\frac{n}{\log(n)}}$. The argument gives insight into the structure of parking functions so we give it in a sequence of simple observations summarized by a theorem at the end. The argument was explained to us by Sourav Chatterjee.

Fix $k$ and $i_1,i_2,\cdots i_k$ in $[n]$. Let $x_j=i_j/n$. For $\pi \in \pfn$, let $F^{\pi}(x)=\frac{1}{n}\sum_{j=1}^n\delta_{\pi_j/n\leq x}$ be the empirical distribution function of $\{\pi_i/n\}$. By symmetry,
\begin{align}\label{eqnew1} P&\left (\pi:\frac{\pi_1}{n}\leq x_1,\cdots,\frac{\pi_k}{n}\leq x_k\right)\\&\hspace{1in}= E\left[\frac{1}{n(n-1)\dots (n-k+1)}\sum_{\substack{j_1,\cdots,j_k\\\text{distinct}}}\delta_{\frac{\pi_{j_1}}{n}\leq x_{1},\cdots, \frac{\pi_{j_k}}{n}\leq x_{k}}\right]\nonumber
\end{align}
The expectation in (\ref{eqnew1}) is over $\pi\in \pfn$. The expression inside the expectation has the following interpretation. Fix $\pi$, pick $k$-coordinates at random from $\{\pi_1,\cdots, \pi_n\}$, sampling without replacement. Because sampling without replacement is close to sampling with replacement, Theorem 13 in \citet{diaconis1980finite} shows
\begin{equation}\label{eqnew2}
\frac{1}{n(n-1)\dots (n-k+1)}\sum_{\substack{j_1,\cdots,j_k\\\text{distinct}}}\delta_{\frac{\pi_{j_1}}{n}\leq x_{1},\cdots, \frac{\pi_{j_k}}{n}\leq x_{k}}=F^\pi(x_1)\dots F^\pi(x_k)+O\left(\frac{k(k-1)}{n}\right)
\end{equation}
The error term in (\ref{eqnew2}) is uniform in $\pi,k,n.$
Since $0\leq F^\pi(x_i)\leq 1$, by \citep[27.5]{billingsley2012probability}
\begin{equation}\label{eqneq3} \left| \prod_{i=1}^n F^\pi(x_i)-\prod_{i=1}^nx_i\right|\leq \sum_{i=1}^k\left|F^\pi(x_i)-x_i\right|
\end{equation}
Let $f=(f_1,\cdots,f_n)$ be a random function from $[n]$ to $[n]$ with $G^f(x)=\frac{1}{n}\sum\delta_{f_i/n\leq x}$. By Hoeffding's inequality for the binomial distribution, for any $\epsilon>0$ and any $x\in [0,1]$,
\begin{equation}
P\left(\left|G^f(x)-x\right|>\epsilon\right)\leq 2 e^{-2\epsilon^2n}.
\end{equation}
Since a random parking function is a random function, conditional on being in $\pf_n$,
\begin{align}\label{eqnew5} P\left(\left|F^\pi(x)-x\right|>\epsilon\right)&= P\left(\left|G^f(x)-x\right|>\epsilon|f\in \pfn\right)=\frac{ P\left(\left|G^f(x)-x\right|>\epsilon \text{ and }f\in \pfn\right)}{P(f\in \pfn)}\\&\nonumber\leq 2e^{-2\epsilon^2n}n\left(1-\frac{1}{n+1}\right)^{n-1}<2ne^{-2\epsilon^2n}.
\end{align}
By a standard identity  \citep[21.9]{billingsley2012probability},
\begin{equation}E[F^\pi(x)-x]\leq 2\sqrt{\frac{\log(n)}{n}} \text{ for } n\geq 4 \text{ uniformly in }x.
\end{equation}
Combining bounds gives
\begin{thm}\label{thmaddedlate}For $\pi$ uniformly chosen in $\pfn$ and $f$ uniformly chosen in $\calfn$, for all $n\geq4$,
\begin{align*}\sup_{x_1,\dots,x_k}&\left|P\left (\pi:\frac{\pi_1}{n}\leq x_1,\cdots,\frac{\pi_k}{n}\leq x_k\right)-P\left (f:\frac{f_1}{n}\leq x_1,\cdots,\frac{f_k}{n}\leq x_k\right)\right|\\&\hspace{2in}\leq 2k\sqrt{\frac{\log(n)}{n}}+\frac{k(k-1)}{n}.\end{align*} The sup is over $x_j=\frac{i_j}{n}$ with $i_j\in [n]$, $1\leq j\leq k$.
\end{thm}
\begin{remark}We are not at all sure about the condition $k\ll \sqrt{n/(\log(n))}$. For the distance used in \ref{thmaddedlate}, the bound in (\ref{eqnew2}) for sampling with\textbackslash without replacement can be improved to $k/n$ using a result of Bobkov \citep{bobkov2005generalized}. The only lower bound we have is $k\ll n$ from events such as having more than one n, more than two $n$s or $n-1$s, etc.  In particular, the probability that a randomly chosen function $f:[k]\rightarrow [n]$ has more than $i$ entries more than $n-i+1$ for some $i$ (and thus cannot possibly be the first $k$ entries of a parking function) is $$1-\frac{(n-k+1)(n+1)^{k-1}}{n^k}.$$  If $k=sn$ for some constant $0<s<1$, the probability goes to $1-(1-s)e^s>0$, so $k\ll n$.
\end{remark}
\section{Voyeurism (using the literature)}\label{sec3}

In this section we enjoy ourselves by looking through the literature on parking functions, trying to find theorems or generating functions which have a simple probabilistic interpretation. We took Catherine Yan's wonderful survey and just began turning the pages. The first ``hit'' was her Corollary 1.3, which states
\begin{equation}
\sum_{\pi\in\pfn}q^{\text{car}\{i:\pi_i=\pi_{i+1}\}}=(q+n)^{n-1}.
\label{31}
\end{equation}
This looks promising; what does it mean? Divide both sides by $(n+1)^{n-1}$. The left side is $E\{q^{R(\pi)}\}$ with $R(\pi)=\text{car}\{i:\pi_i=\pi_{i+1}\}$ the number of repeats in $\pi$. The right side is the generating function of $S_{n-1}=X_1+\dots+X_{n-1}$, with $X_i$ independent and identically distributed, with
\begin{equation*}
X_1=\begin{cases}1&\text{probability $\nicefrac1{(n+1)}$}\\0&\text{probability $\nicefrac{n}{(n+1)}$.}\end{cases}
\end{equation*}
A classical theorem \citep[p.~286]{feller} shows that $P_{n-1}(j)=P(X_1+\dots+X_{n-1}=j)$ has an approximate Poisson distribution $Q(j)=\nicefrac1{ej!}$,
\begin{equation*}
P_n(j)\sim\frac1{(ej!)}.
\end{equation*}
Indeed,
\begin{equation*}
\|P_n-Q\|_{\text{TV}}\leq\frac98\frac{(n-1)}{(n+1)^2}.
\end{equation*}
Putting things together gives
\begin{thm}
Pick $\pi\in\pfn$ uniformly. Let $R(\pi)$ be the number of repeats in $\pi$ reading left to right. Then for fixed $j$ and $n$ large,
\begin{equation*}
P\{R(\pi)=j\}\sim\frac1{(ej!)}.
\end{equation*}
\end{thm}

\begin{remark}
How does our analogy with random functions play out here? Almost perfectly. If $f=(f_1,\dots,f_n)$ is a function $f:[n]\to[n]$, let
\begin{equation*}
R(f)=\sum_{i=1}^{n-1}Y_i(f)\qquad\text{with }Y_i(f)=\begin{cases}1&\text{if $f_i=f_{i+1}$}\\0&\text{otherwise.}\end{cases}
\end{equation*}
It is easy to see that if $f$ is chosen uniformly at random then the $Y_i$ are independent with
\begin{equation*}
Y_i(f)=\begin{cases}1&\text{probability $\nicefrac1{n}$}\\0&\text{probability $\nicefrac{(n-1)}{n}$.}\end{cases}
\end{equation*}
The Poisson limit theorem says
\end{remark}

\begin{thm}
For $f\in\calfn$ chosen uniformly, for fixed $j$ and large $n$,
\begin{equation*}
P\{R(f)=j\}\sim\frac1{ej!}.
\end{equation*}
\end{thm}

The next result that leapt out at us was in \citet[Sect.~1.2.2]{yan}. If $\pi\in\pfn$ is a parking function, say that car $i$ is ``lucky'' if it gets to park in spot $\pi_i$. Let $L(\pi)$ be the number of lucky cars. \citet{gessel2006refinement} give the generating function
\begin{equation*}
\sum_{\pi\in\pfn}q^{L(\pi)}=q\prod_{i=1}^{n-1}\left[i+(n-i+1)q\right].
\end{equation*}
As above, any time we see a product, something is independent. Dividing both sides by $(n+1)^{n-1}$, the right side is the generating function of $S_n=1+\sum_{i=1}^{n-1}X_i$ with $X_i$ independent,
\begin{equation*}
X_i=\begin{cases}0&\text{probability $\nicefrac{i}{(n+1)}$}\\1&\text{probability $1-\nicefrac{i}{(n+1)}$.}\end{cases}
\end{equation*}
By elementary probability, the mean and variance of $S_n$ are
\begin{align*}
\mu_n&=1+\sum_{i=1}^{n-1}\left(1-\frac{i}{n+1}\right)=n-\frac12\frac{n(n-1)}{n+1}\sim\frac{n}2,\\
\sigma_n^2&=\sum_{i=1}^{n-1}\frac{i}{n+1}\left(1-\frac{i}{n+1}\right)\sim\frac{n}{6}.
\end{align*}
Now, the central limit theorem \citep[p.~262]{feller} applies.

\begin{thm}
For $\pi\in\pfn$ chosen uniformly, let $L(\pi)$ be the number of lucky $i$. Then for any fixed real $x$, $-\infty<x<\infty$, if $n$ is large,
\begin{equation*}
P\left\{\frac{L(\pi)-\nicefrac{n}2}{\sqrt{\nicefrac{n}6}}\leq x\right\}\sim\int_{-\infty}^x\frac{e^{-t^2/2}}{\sqrt{2\pi}}\ dt.
\end{equation*}
\end{thm}

One more simple example: Let $N_1(\pi)=\text{car}\{i:\pi_i=1\}$. \citet[Cor.~1.16]{yan} gives
\begin{equation*}
\sum_{\pi\in\pfn}q^{N_1(\pi)}=q(q+n)^{n-1}.
\end{equation*}
Dividing through by $(n+1)^{n-1}$ shows that $N_1(\pi)$ has exactly the same distribution as of $1+X$ where $X$ has a binomial($n-1,\nicefrac1{(n+1)}$) distribution. As above, a Poisson limit holds.

\begin{thm}
Let $\pi\in\pfn$ be chosen uniformly. Then $Z(\pi)$, the number of ones in $\pi$, satisfies
\begin{equation*}
P\{Z(\pi)=1+j\}\sim\frac1{ej!}
\end{equation*}
for $j$ fixed and $n$ large.
\end{thm}

\begin{remark}
Any $\pi\in\pfn$ has $N_1(\pi)\geq1$ and $\pi_i\equiv1$ is a parking function. The theorem shows that when $n$ is large, $N_1(\pi)=1$ with probability $\nicefrac1{e}$ and $$E(N_1(\pi))=\frac{2n}{(n+1)}\sim 2.$$ A random \textit{function} has $N_1(f)$ having a limiting Poisson(1) distribution; just 1 off $N_1(\pi)$. A parking function can have at most one $i$ with $\pi_i=n$. Let $N_i(\pi)$ be the number of $i$'s in $\pi$. An easy argument shows $P(N_n(\pi)=1)\sim\nicefrac1{e}$, $P(N_n(\pi)=0)\sim1-\nicefrac1{e}$. A random function has $N_i(f)$ with an approximate Poisson(1) distribution for any $i$. Interestingly, note that this is already an example of a distinct difference between $\pfn$ and $\calfn$: while $P(N_n(f)=1)\sim\nicefrac1{e}$, $P(N_n(f)=0)=\nicefrac1{e}$ in contrast to the parking function case.
\end{remark}

\section{Equality of ensembles}\label{sec4}

Previously we have found the same limiting distributions for various features $T(\pi)$ and $T(f)$ with $\pi$ random in parking functions $\pfn$ and $f$ random in all functions $\calfn$. This section gives collections of features where the two ensembles have exactly the same distribution for fixed $n$. This includes the descent, ascent, and equality processes. These have a rich classical structure involving a determinantal point process, explained below.

To get an exact equivalence it is necessary to slightly change ensembles. Let $\tcalfn=\{f:[n]\to[n+1]\}$. Thus $|\tcalfn|=(n+1)^n$, $\pfn\subseteq\tcalfn$, and $|\pfn|/|\tcalfn|=\nicefrac1{(n+1)}$.One might ask here why we ever consider $\tilde{\calfn}$ rather than the more seemingly natural $\calfn$. Elements of $\pfn$ are naturally selected uniformly at random by first uniformly selecting element from $\tilde\calfn$ and then doing a bit of additional work to find a true parking function. In practice, this means that while we can get asymptotic results using $\pfn$ and $\calfn$, we can generally get precise results for all $n$ using $\tilde\calfn$. In practice, this also tells us about the asymptotics when we compare $\pfn$ and $\calfn$, since comparing $\calfn$ and $\tilde{\calfn}$ is generally quite easy.

For $f\in\tcalfn$ let $X_i(f)=1$ if $f_{i+1}<f_i$ and 0 otherwise, so $X_1,X_2,\dots,X_{n-1}$ give the descent pattern in $f$. The following theorem shows that the descent pattern is the same on $\pfn$ and $\tcalfn$.

\begin{thm}
Let $\pi\in\pfn$ and $f\in\tcalfn$ be uniformly chosen. Then
\begin{equation*}
P\left\{X_1(\pi)=t_1,\dots,X_{n-1}(\pi)=t_{n-1}\right\}=P\left\{X_1(f)=t_1,\dots,X_{n-1}(f)=t_{n-1}\right\}
\end{equation*}
for all $n\geq2$ and $t_1,\dots,t_{n-1}\in\{0,1\}$.
\label{thm7}
\end{thm}

The same theorem holds with $X_i$ replaced by
\begin{align*}
Y_i(f)&=\begin{cases}1&\text{if $f_{i+1}=f_i$}\\0&\text{otherwise,}\end{cases}\\
W_i(f)&=\begin{cases}1&\text{if $f_{i+1}\leq f_i$}\\0&\text{otherwise,}\end{cases}
\end{align*}
or for the analogs of $X_i,W_i$ with inequalities reversed.

For equalities, the $\{Y_i(f)\}_{i=1}^{n-1}$ process is independent and identically distributed with $P(Y_i=1)=\nicefrac1{(n+1)}$. Thus any fluctuation theorem for independent variables holds for $\{Y_i(\pi)\}_{i=1}^{n-1}$. In particular, this explains \eqref{31} above.

The descent pattern in a random sequence is carefully studied in \citet{pd-borodin}. Here is a selection of facts, translated to the present setting. Throughout, $X_i=X_i(\pi)$, $1\leq i\leq n-1$, is the descent pattern in a random function in $\tcalfn$ (and thus by the previous theorem a descent pattern in a random parking function).

\paragraph{Single descents}
\begin{equation}
P(X_i=1)=\frac12-\frac1{2(n+1)}.
\end{equation}

\paragraph{Run of descents}
For any $i,j$ with $1\leq i+j\leq n$,
\begin{equation}
P(X_i=X_{i+1}=\dots=X_{i+j-1}=1)=\binom{n+1}{j}\bigg/(n+1)^{j}.
\end{equation}
In particular,
\begin{equation*}
\text{Cov}(X_iX_{i+1})=E(X_iX_{i+1})-E(X_i)E(X_{i+1})=-\frac1{12}\left(1-\frac1{(n+1)^2}\right).
\end{equation*}

\paragraph{Stationary one-dependence}
The distribution of $\{X_i\}_{i\in[n-1]}$ is stationary: for $J\subseteq[n-1]$, $i\in[n-1]$ with $i+J\subseteq[n-1]$, the distribution of $\{X_j\}_{j\in J}$ is the same as the distribution of $\{X_j\}_{j\in i+J}$. Further, the distribution of $\{X_i\}_{i\in[n-1]}$ is one-dependent: if $J\subseteq[n-1]$ has $j_1,j_2\in J\Rightarrow|j_1-j_2|>1$, then $\{X_j\}_{j\in J}$ are jointly independent binary random variables with
\begin{equation}
P(X_j=1)=\frac12-\frac1{2(n+1)}.
\end{equation}
The following central limit theorem holds.

\begin{thm}
For $n\geq2$, $S_{n-1}=X_1+\dots+X_{n-1}$ has
\begin{align*}
\text{mean}&\qquad (n-1)\left(\frac12-\frac1{2(n+1)}\right),\\
\text{variance}&\qquad \frac{(n+1)}{12}\left(1-\frac1{(n+1)^2}\right),
\end{align*}
and, normalized by its mean and variance, $S_{n-1}$ has a standard normal limiting distribution for large $n$.
\end{thm}

\paragraph{$k$-point correlations}
For $A\subseteq[n-1]$,
\begin{equation}
P(X_i=1\text{ for }i\in A)=\prod_{i=1}^k\binom{n+1}{a_i+1}\bigg/(n+1)^{a_i+1}
\end{equation}
if $A=\cup A_i$, with $|A_i|=a_i$ and $A_i$ disjoint, nonempty consecutive blocks, e.g., $A=\{2,3,5,6,7,11\}=\{2,3\}\cup\{5,6,7\}\cup\{11\}$. (It has $a_1=2$, $a_2=3$, and $a_3=1$.)

\paragraph{Determinant formula}
Let $\epsilon_1,\epsilon_2,\dots,\epsilon_{n-1}\in\{0,1\}$ have exactly $k$ ones in positions $s_1<s_2<\dots<s_k$.
\begin{equation}
P\{X_1=\epsilon_1,\dots,X_n=\epsilon_n\}=\frac1{(n+1)^n}\det\binom{s_{j+1}-s_i+n}{n}.
\end{equation}
The determinant is of a $(k+1)\times(k+1)$ matrix with $(i,j)$ entry $\binom{s_{j+1}-s_i+n}{n}$ for $0\leq i,j\leq n$ with $s_0=0$, $s_{k+1}=n$.

In \citet{pd-borodin}, these facts are used to prove that $\{X_i\}_{i\in[n-1]}$ is a determinantal point process and a host of further theorems are given.

The proof of \ref{thm7} proves somewhat more. Let $P$ be a partial order on $[n]$ formed from a disjoint union of chains:
\begin{center}
\includegraphics[keepaspectratio,scale=0.4,clip]{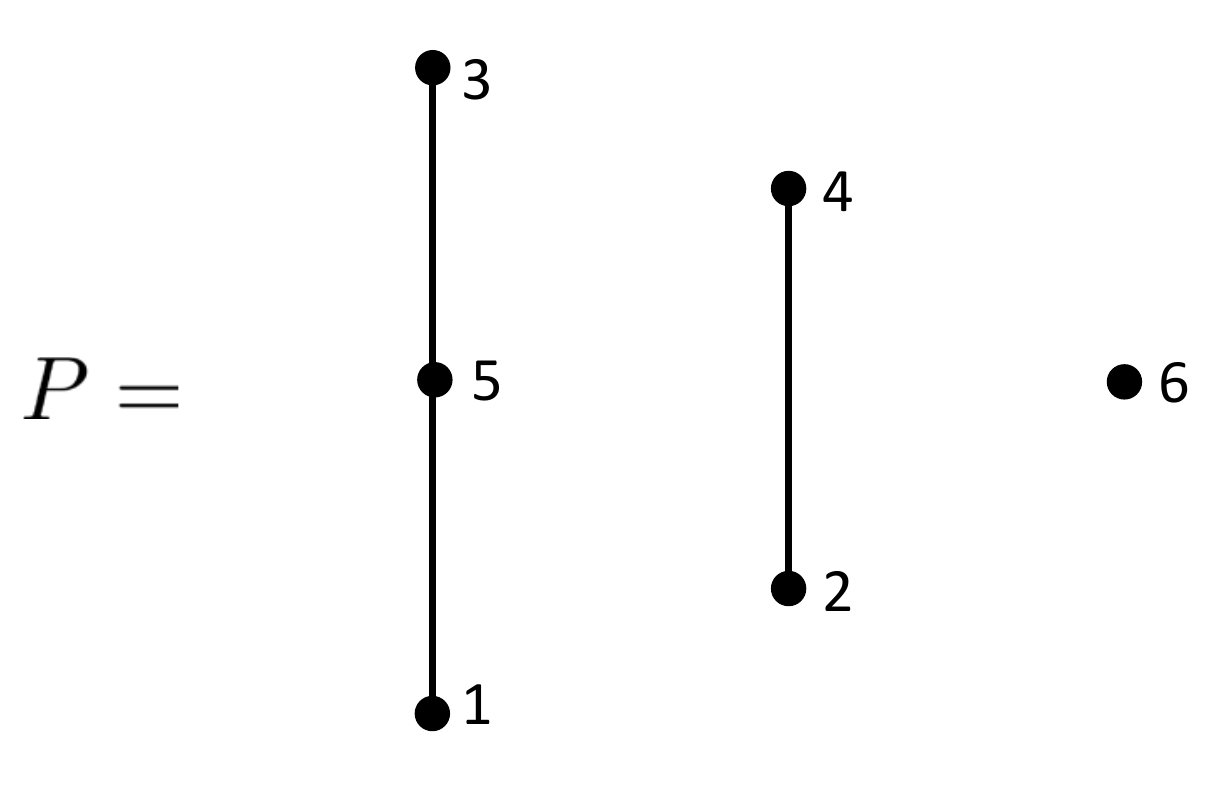}
\end{center}

\begin{defn}
A function $f\in\tcalfn$ is \textit{$P$-monotone} if $i<_pj$ implies $f(i)<f(j)$.
\end{defn}

\begin{thm}
Let $P$ be a poset on $[n]$ formed from disjoint chains. Then, for $\pi\in\pfn$, $f\in\tcalfn$ uniformly chosen,
\begin{equation*}
P\{\text{$\pi$ is $P$-monotone}\}=P\{\text{$f$ is $P$-monotone}\}.
\end{equation*}
\label{thm9}
\end{thm}

\begin{proof}
Let $f\in\tcalfn$ be $P$-monotone. Consider the set
\begin{equation}
S(f)=\{f+_{n+1}k(1,\dots,1)\}
\label{46}
\end{equation}
when $+_{n+1}$ indicates addition $\mod{n+1}$ and representatives of each equivalence class are chosen from $[n]$. Note that $f'\in S(f)$ need not be $P$-monotone, e.g., the addition $\mod{n+1}$ may cause the largest element to be smallest. Simply selecting the numbers corresponding to each chain in $P$ and reordering them to be $P$-monotone within each chain results in a new element of $\tcalfn$ consistent with $P$. Let $S'(f)$ be the reordered functions. Note that if $k$ in \eqref{46} is known, $f$ can be uniquely reconstructed from any element in $S'(f)$. Finally observe that exactly one element of $S'(f)$ is in $\pfn$. Since $|\tcalfn|=(n+1)|\pfn|$, this implies the result.
\end{proof}

\begin{example}
With $n=4$,
\begin{center}
\includegraphics[keepaspectratio,scale=0.4,clip]{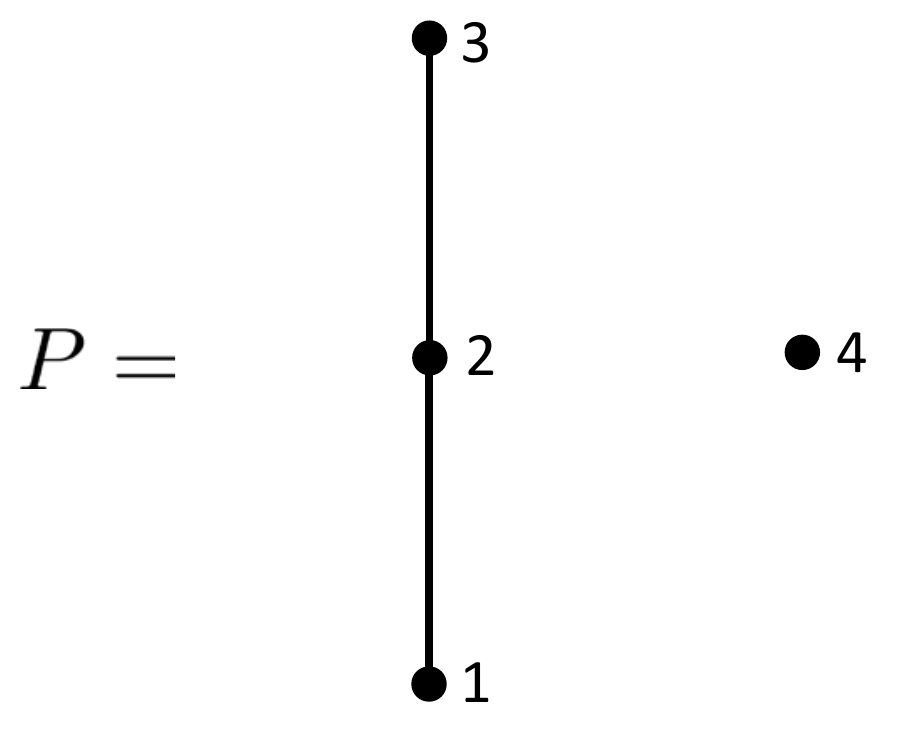}
\end{center}
and $f=(3,4,5,1)$, $f$ is $P$-monotone but $S(f)$ contains $(4,5,1,2)$ which is not $P$-monotone. However, $(1,4,5,2)$ is $P$-monotone, and would be used as a new element of $\tcalfn$ consistent with $P$.
\end{example}

\ref{thm9} can be generalized. For every chain $c$ in a poset $P$ of disjoint chains, let $<_c$ be one of $<,>,\leq ,\geq ,=$. Say $a<_{P_c}b$ if $a<_P b$ and $a,b\in c$. Say $f\in\tcalfn$ is $P$-chain monotone if for all chains $c$ in $P$ and $i<_{P_c}j$ implies $f(i)<_cf(j)$. Thus we assign to disjoint subsets of $[n]$ the requirement that $f$ is either (weakly) decreasing or (weakly) increasing on the subset in some order. The argument for \ref{thm9} gives
\begin{thm}
Let $P$ be a poset on $[n]$ formed from a union of disjoint chains with order $P_c$ specified on each chain as above. Let $\pi\in\pfn$, $f\in\tcalfn$ be uniformly chosen. Then
\begin{equation*}
P\{\text{$\pi$ is $P$-chain monotone}\}=P\{\text{$f$ is $P$-chain monotone}\}.
\end{equation*}
\label{thm10}
\end{thm}

\begin{remarks}
\begin{enumerate}
\item Theorems 9, 10 imply Theorem 7: For $f$ in $\tcalfn$ let $$\operatorname{Des}(f) =\{i
\in [n-1] |f(i+1) <f(i)\}.$$ By inclusion-exclusion, it is enough
to show that for any $S $ contained in $[n-1]$, $P(S\subset \operatorname{Des}(f) ) =
P(S\subset\operatorname{Des}(\pi))$. For every $i$ in $S$, say that $i-1<_P i+1$. These form the covering relations of the poset $P$; if $S = \{1,3,4,7,9,10\}$ the poset is the following:

\begin{center}
\begin{tikzpicture}[scale=1,every node/.style={circle,fill=black,inner sep=0pt,minimum size=.1cm}]
\node[label=right:{$1$}] (1) at (0,0) {};
\node[label=right:{$2$}] (2) at (1,-1.5) {};
\node[label=right:{$3$}] (3) at (1,0) {};
\node[label=right:{$4$}] (4) at (1,1.5) {};
\node[label=right:{$5$}] (5) at (2,0) {};
\node[label=right:{$6$}] (6) at (3,-.75) {};
\node[label=right:{$7$}] (7) at (3,.75) {};
\node[label=right:{$8$}] (8) at (4,-1.5) {};
\node[label=right:{$9$}] (9) at (4,0) {};
\node[label=right:{$10$}] (10) at (4,1.5) {};
\draw (2) -- (3);
\draw (3) -- (4);
\draw (6) -- (7);
\draw (8) -- (9) -- (10);
\end{tikzpicture}
\end{center}
\noindent The f's that are $P_S$ monotone (with respect to $>$) are
exactly the f's with $S\subset\operatorname{Des}(f)$.
\item \ref{thm10} can be combined with inclusion/exclusion to give further results. For example, considering
\begin{center}
\begin{tikzpicture}[scale=1,every node/.style={circle,fill=black,inner sep=0pt,minimum size=.1cm}]

\node[label=right:{$1$}] (1) at (0,-.75) {};
\node[label=right:{$2$},label=above:{$<$}] (2) at (0,.75) {};
\node[label=right:{$3$}] (3) at (1,0) {};
\node[draw=none, fill=white, label=right:{$P_1:$}] (P) at (-1,0) {};
\draw (1)-- (2) ;
\node[label=right:{$1$}] (n1) at (3,-1.5) {};
\node[label=right:{$2$}] (n2) at (3,0) {};
\node[label=right:{$3$},label=above:{$<$}] (n3) at (3,1.5) {};
\node[draw=none, fill=white, label=right:{$P_2:$}] (nP) at (2,0) {};
\draw (n1)-- (n2) -- (n3);
\end{tikzpicture}
\end{center}
shows that the chance of $f(1)< f(2)\geq f(3)$ as the same for both $\pfn$ and $\widehat{\pf}_n$. Call such an occurrence a ``weak peak at position 2''.  Weak peaks at position $i$ are the same for randomly chosen $f$ and $\pi$.

\item Results for longest consecutive increasing sequences have the same distribution in $\pfn$ and $\tcalfn$. For example:
\begin{center}
\begin{tikzpicture}[scale=1,every node/.style={circle,fill=black,inner sep=0pt,minimum size=.1cm}]
\node[label=right:{$1$}] (1) at (0,-2.25) {};
\node[label=right:{$2$}] (2) at (0,-.75) {};
\node[label=right:{$3$}] (3) at (0,.75) {};
\node[label=right:{$4$},label=above:{$<$}] (4) at (0,2.25) {};
\node[label=right:{$5$}] (5) at (1,0) {};
\node[label=right:{$6$}] (6) at (2,0) {};
\node[label=right:{$7$}] (7) at (3,-.75) {};
\node[label=right:{$8$},label=above:{$>$}] (8) at (3,.75) {};
\node[draw=none, fill=white, label=right:{$P_1:$}] (P) at (-1,0) {};
\draw (1)-- (2) -- (3)--(4);
\draw (7) -- (8);
\node[label=right:{$1$}] (n1) at (5,-3) {};
\node[label=right:{$2$}] (n2) at (5,-1.5) {};
\node[label=right:{$3$}] (n3) at (5,0) {};
\node[label=right:{$4$}] (n4) at (5,1.5) {};
\node[label=right:{$5$},label=above:{$<$}] (n5) at (5,3) {};
\node[label=right:{$6$}] (n6) at (6,0) {};
\node[label=right:{$7$}] (n7) at (7,-.75) {};
\node[label=right:{$8$},label=above:{$>$}] (n8) at (7,.75) {};
\node[draw=none, fill=white, label=right:{$P_2:$}] (nP) at (4,0) {};
\draw (n1)-- (n2) -- (n3)--(n4)--(n5);
\draw (n7) -- (n8);
\end{tikzpicture}
\end{center}
Functions consistent with $(P_1<_c)$ and not $(P_2<_c)$ have longest increasing sequences of length 4, e.g., $f=(1,2,3,4,4,5,6,5)$.

\item The descent patterns above appear as local properties but they also yield global properties such as the total number of descents (or equalities,\textellipsis) being equidistributed.

\item We tried a number of generalizations of the patterns above which failed to be equidistributed. This includes peaks, $f(1)<f(2)>f(3)$, or mixed orderings within a chain, $f(1)\leq f(2)<f(3)$. Equidistribution also failed for $P$ made from non-disjoint chains, such as
\begin{center}
\begin{tikzpicture}[scale=.75,every node/.style={circle,fill=black,inner sep=0pt,minimum size=.1cm}]
\node[label=right:{$1$}] (1) at (-1,-1.5) {};
\node[label=right:{$2$}] (2) at (0,0) {};
\node[label=right:{$3$}] (3) at (-1,1.5) {};
\node[label=right:{$4$}] (4) at (1,-1.5) {};
\node[label=right:{$5$}] (5) at (1,1.5) {};

\draw(1)-- (2) -- (3);
\draw (4) -- (2)--(5);
\end{tikzpicture}
\end{center}
or forced differences larger than one, e.g., $f(1)<f(2)-1$. Of course, asymptotic equidistribution may well hold in some of these cases.
\end{enumerate}
\end{remarks}
\begin{remark} Note that \ref{cor2} already generally shows that we should not expect an equality of ensembles for statistics (such as the number of $\pi_i=1$) that are computed based on the values of the cars, even in the limit. \ref{thm10} gives us that we should generally expect an equality of ensembles for certain types of statistics that are based on the relative values of cars for every $n$. We do not have an example of a statistic that cannot be derived from \ref{thm10} that does ``show an equivalence of ensembles,'' either for every $n$ or in the limit, and would be quite interested in such an example. An open question is simply to describe in some natural way the set of statistics that can be shown to demonstrate our desired equality of ensembles by \ref{thm10}.
\end{remark}
There is another way to see the above equivalence of ensembles, as kindly remarked by Richard Stanley.
\begin{defn}[species]
Let $\mu_i$ be the number of entries in a function $f:[n]\rightarrow[n+1]$ which occur exactly $i$ times.  (In particular, let $\mu_0$ give the number of values in $[n+1]$ which do not occur in $f$, so that $n+1-\mu_0$ gives the number of distinct values taken on by $f$.)  Then call  $\mu(f)=(\mu_0,\mu_1,\dots\mu_n)$ the species of $f$.
\end{defn}
\begin{example}Let $f:[6]\rightarrow[7]$ such that $(f[1],\cdots, f[6])=(2,1,2,1,1,1)$. The species of $f$ is $(5,0,1,0,1,0,0)$.
\end{example}
Note that the species of $\pi$ is the same as the species of $\pi+_{n+1} (1,\dots,1)$.  Thus by a similar argument to above, the distribution of species of functions in $\pfn$ and $\tcalfn$ are equidistributed. (Stanley kindly noted this can also be derived from the one variable Frobenius characteristic of the diagonal harmonics, when expressed in terms of the homogeneous basis.  The relationship to this space will be explained in \ref{sec6}.)  Whether or not a function fits a poset $P$ of disjoint chains (as defined above) depends only on its species.  This follows from the theory of P-partitions, since in this language, functions which fit a poset P are order preserving or order reversing  maps of $P$ (depending on the choice of $\leq$ or $\geq$)  or the strictly ordering preserving or order reversing (depending on the choice of $<$ or $>$) maps of $P$.  See \citet[p.\ 211]{stanley1997enumerative} for an introduction to P-partitions.

There has been extensive work on the distribution of species counts in a random function.  The following list gives a brief description of available distributional results.  They are stated for random functions but of course, they apply to random parking functions.  

Let $b$ balls be allocated, uniformly and independently, into $B$ boxes.  Let $\mu_r(b,B)=\mu_r$ be the number of boxes containing $r$ balls.  Thus
\begin{equation}\label{restrictions} \sum_{r=0}^b \mu_r=B \text{ and }\sum_{r=1}^b r\mu_r=b
\end{equation}
The joint distribution of $\{\mu_r\}_{r=0}^b$ is a classical subject in probability.  It is the focus of \citet[Ch. 2]{kolchin1978random} which contains extensive historical references.  The following basics give a feeling for the subject.  Throughout, we take $B=n+1$ and $b=n-1$ as this is the object of interest for $\pfn$ and $\tcalfn$.
\begin{enumerate}
\item The joint distribution is specified exactly by \citet[p.36]{kolchin1978random}:
$$P\left\{\mu_n=m_n,\, 0\leq r\leq b\right\}=\begin{cases}\frac{B!b!}{B^b\prod_{r=0}^b r!^{m_r}m_r!}&\text{ if (\ref{restrictions}) is satisfied}\\0& \text{ else}.
\end{cases}$$
\item The means and covariances are:
\begin{align*}E(\mu_r)&=B{b\choose r} \frac{1}{B^r}\left(1-\frac{1}{B}\right)^{b-r}\\
E(\mu_r^2)&=E(\mu_r)+B(B-1)\frac{b^{[2r]}}{r!^2B^{2r}}\left(1-\frac{2}{B}\right)^{b-2r}\\
E(\mu_r\mu_t)&=B(B-1)\frac{b^{[r+t]}}{r!t!N^{r+t}}\left(1-\frac{2}{N}\right)^{b-r-t}
\end{align*}
where $x^{[r]}=x(x-1)\dots(x-r+1)$.
\item Asymptotically, as $n\rightarrow \infty$, for fixed $r,t$,

$$E(\mu_r)=\frac{n}{er!}+\frac1{er!}\left(r-\frac12-{r\choose 2}\right)+O\left(\frac1n\right)$$
$$Cov(\mu_r,\mu_t)\sim n \sigma_{rt}\text{, }$$ $$\sigma_{rr}=\frac1{er!}\left(1-\frac1{er!}-\frac1{er!}(1-r)^2\right)\text{, }\sigma_{rt}=\frac{1}{e^2r!t!}(1+(1-r)(1-t)).$$
\item For any fixed $s$ and $0\leq r_1<r_2\dots<r_s$, the random vector
$$X=\left(\frac{\mu_{r_1}-E(\mu_{r_1})}{\sqrt{n}},\frac{\mu_{r_2}-E(\mu_{r_2})}{\sqrt{n}},\cdots, \frac{\mu_{r_s}-E(\mu_{r_s})}{\sqrt{n}}\right)$$
has limiting covariance matrix $\Sigma=(\sigma_{r_ir_j})_{1\leq i,j\leq s}$ with $\sigma_{ij}$ as defined just above.  This matrix can be shown to be positive definite and \citet[p. 54]{kolchin1978random} imply that the vector has a limiting normal approximation.

If $D^2=\operatorname{det}(\Sigma)$, they show 
$$P(X\in G)=\frac{1}{(2\pi)^{s/2}}\int_G e^{-\frac{1}{2D^2}\sum_{i,j=1}^s\Sigma_{i,j}\mu_i\mu_j} d\mu_1\dots d\mu_s+o(1).$$
The also give local limit theorems and rates of convergence.
\item For $n-1$ balls dropped into $N+1$ boxes, the individual box counts are well approximated by independent Poisson(1) random variables.\footnote{We switch to function values here, for a moment, which are of course not equidistributed between $\tcalfn$ and $\pfn$, but it will be translated into species in a moment.}
  In particular, the maximal box count $m_n$ is distributed as the maximum of $n$ independent Poisson(1) variables.  This is well known to concentrate on one of two values with slowly oscillating probabilities.  \citet[Ch.\ 2]{kolchin1986random} show that if $r=r(n)$ is chosen so that $r>1$ and $\frac{n}{er!}\rightarrow \lambda$ where $\lambda >0$, then 
$$P(\max m_n=r-1)\rightarrow e^{-\lambda}\text{, }  P(m_n=r)\rightarrow 1-e^{-\lambda}$$
Very roughly, $m_n\approx \frac{\log n}{\log\log n}.$  See \citet{briggs2009note}.  This determines the largest $r$ such that $\mu_r>0.$
\item There are many further properties of random functions $f:[n]\rightarrow[n+1]$ known.  See \citet{ald-pit}and \citet{kolchin1978random}.  We mention one further result, consider such a random function and let $l(f)$ be the length of the longest increasing subsequence as $f$ is read in order $f(1), f(2),\cdots, f(n)$.  This has the same asymptotic distribution as the length of the longest increasing subsequence in a random permutation.  Indeed the asymptotics of the shape of the tableau of $f$ under the RSK algorithm matches that of a random permutation.  In particular, $$P\left(\frac{l(f)-2\sqrt{n}}{n^{1/6}}\leq x\right)\rightarrow F(x)$$ with $F(x)$ the Tracy-Widom Distribution.  See \citet{baik1999distribution}.
\item One particularly well known statistic that is equidistributed on $\tcalfn$ and $\pfn$ and clearly follows from their equidistribution of species is the number of inversions: $\{f(i)>f(j):i<j\}$. 
\section{From probability to combinatorics}\label{sec5}

\end{enumerate}
For $\pi\in\pfn$ and $0\leq x\leq 1$, let
\begin{equation*}
F^\pi(x)=\frac1{n},\,\#\{i:\pi_i\leq nx\}.
\label{51}
\end{equation*}
From the definitions, $F^\pi(x)\geq x$ for $x=\nicefrac{i}{n}$, $0\leq i\leq n$. The main result of this section studies $\{F^\pi(x)-x\}_{0\leq x\leq 1}$ as a stochastic process when $\pi$ is chosen uniformly. It is shown that $\sqrt{n}\{F^\pi(x)-x\}_{0\leq x\leq 1}$ converges to the Brownian excursion process $\{E_x\}_{0\leq x\leq 1}$. This last is well studied and the distribution of a variety of functions are available. One feature of these results: they show a deviation between parking functions $\pi$ and all functions $f$. For a random function, $\sqrt{n}\{F^f(x)-x\}_{0\leq x\leq 1}$ converges to the Brownian bridge $\{B_x^0\}_{0\leq x\leq 1}$. This has different distributions for the functions of interest. We state results for three functions of interest followed by proofs.

\begin{thm}[Coordinate counts]
For $0<x<1$ fixed,
\begin{equation}
\frac{\#\{i:\pi_i<nx\}-nx}{\sqrt{n}}\Longrightarrow G_x,
\label{52}
\end{equation}
with $G_x$ a random variable on $[0,\infty)$ having
\begin{equation*}
P\{G_x\leq t\}=\int_0^t\frac{e^{-y^2/2x(1-x)}y^2}{\sqrt{2\pi x^3(1-x)^3}}\ dy.
\end{equation*}
\label{thm11}
\end{thm}

\begin{remark}
$G_x$ is the square of a Gamma(3) random variable scaled by $x(1-x)$. For a random \textit{function} $f$, a similar limit theorem holds with $G_x$ replaced by a normal random variable. 
\end{remark}

\begin{thm}[Maximum discrepancy]
\begin{equation}
\max_{1\leq k\leq n}\frac{\#\{i:\pi_i\leq k\}-k}{\sqrt{n}}\Longrightarrow M,
\label{53}
\end{equation}
with
\begin{equation*}
P\{M\leq t\}=\sum_{-\infty<k<\infty}(1-4k^2t^2)e^{-2k^2t^2}.
\end{equation*}
\label{thm12}
\end{thm}

\begin{remark}
The random variable $M$ has a curious connection to number theory. Its moments are
\begin{equation*}
E(M)=\sqrt{\frac{\pi}2},\quad E(M^s)=2^{-s/2}s(s-1)\Gamma\left(\frac{s}2\right)\zeta(s)=\xi(s),\qquad1<s<2.
\end{equation*}
The function $\xi(s)$ was introduced by Riemann. It satisfies the functional equation $\xi(s)=\xi(1-s)$. See \citet{edwards} or \citet{smith1988honest}. For a random function $f$, a similar limit theorem holds with $M$ replaced by $M_1$, where
\begin{equation*}
P\{M_1\leq t\}=1-e^{-2t^2}.
\end{equation*}
\end{remark}

\begin{thm}[Area]
\begin{equation}
\frac1{\sqrt{n}}\left(\frac{n^2}{2}-\sum_{i=1}^n\pi_i\right)\Longrightarrow A,
\label{54}
\end{equation}
where $A$ has density function
\begin{equation*}
f(x)=\frac{2\sqrt{6}}{x^{10/3}}\sum_{k=1}^\infty e^{-b_k/x^2}b_k^{2/3}U(-5/6,4/3,b_k/x^2)
\end{equation*}
where $U(a,b,z)$ is the confluent hypergeometric function, $b_k=-2\alpha_k^3/27$,  and $a_k$ are the zeros of the Airy function
\begin{gather*}
\textup{Ai}(x)=\frac1{\pi}\int_0^\infty\cos\left(\frac{t^3}3+xt\right)\ dt,\\
a_1=-2.3381, a_2=-4.0879, a_3=-5.5204, \dots, a_j\sim\left(\frac{3\pi}2\right)^{2/3}j^{2/3},
\end{gather*}
\label{thm13}
\end{thm}


\begin{remarks}

\begin{itemize}
\item A histogram of the area based on 50,000 samples from $\pf_{100}$ is shown in Figure 2, while Figure 3 shows (a scaled version of) the limiting approximation. Motivation for studying area of parking functions comes from both parking functions and Macdonald polynomials; see \ref{sec6}.
\item The density $f(x)$ is called the Airy density. It is shown to occur in a host of further problems in \citet{MR2151223} which also develops it's history and properties.  Specific commentary on area of parking functions and the Airy distribution is available also in \cite{flajolet1998analysis}.
\item For a random function $f$, the right side of the theorem is replaced by $A_1$, with $A_1$ a normal $(0,\nicefrac1{12})$ random variable. 
\end{itemize}
\end{remarks}
\begin{figure}\begin{center}
\includegraphics[width=4in]{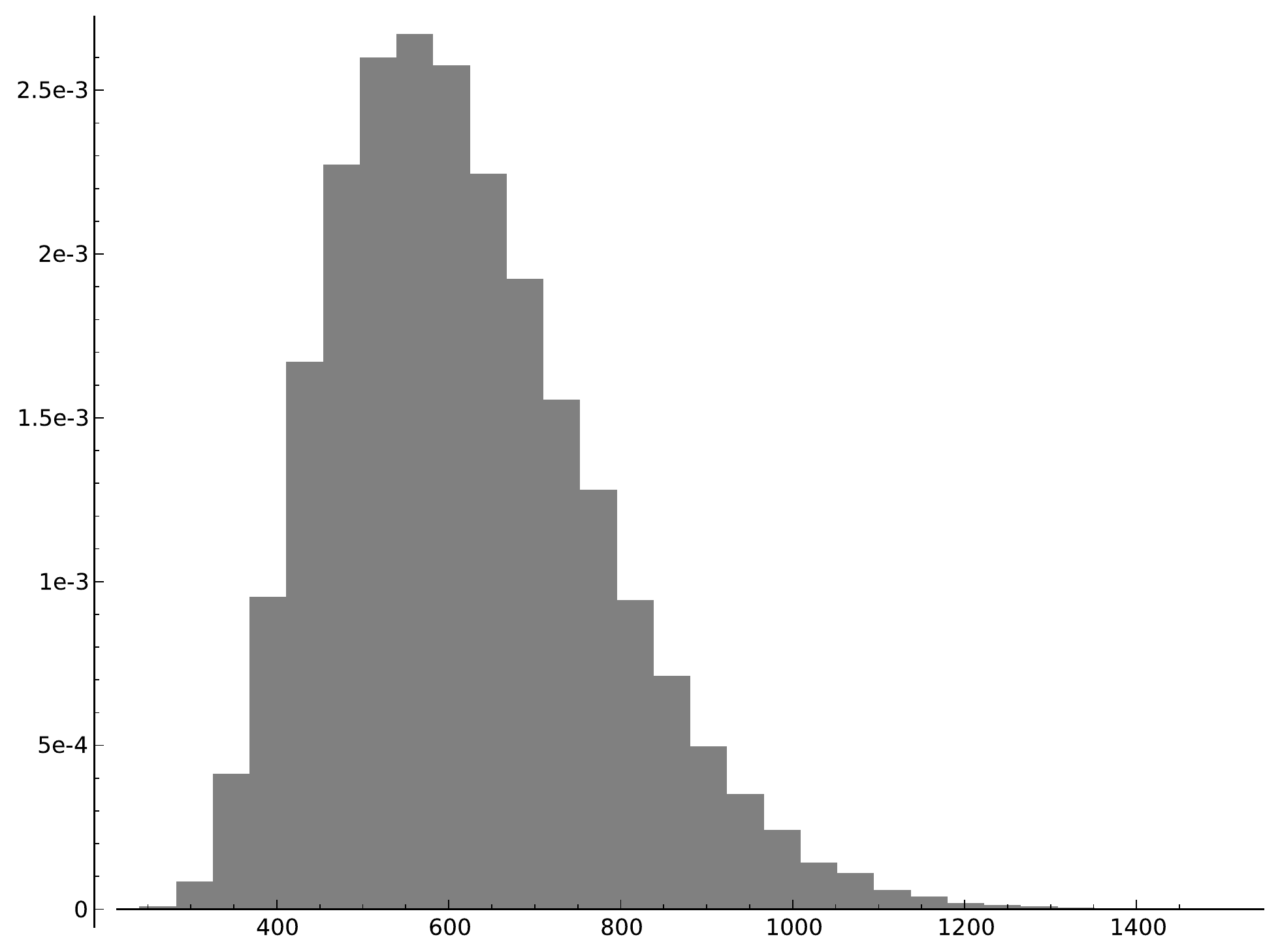}\end{center}
\caption{The area of 50,000 parking functions of size 100.}
\label{fig2}
\end{figure}
\begin{figure}[htb]
\begin{center}
\includegraphics[width=3in]{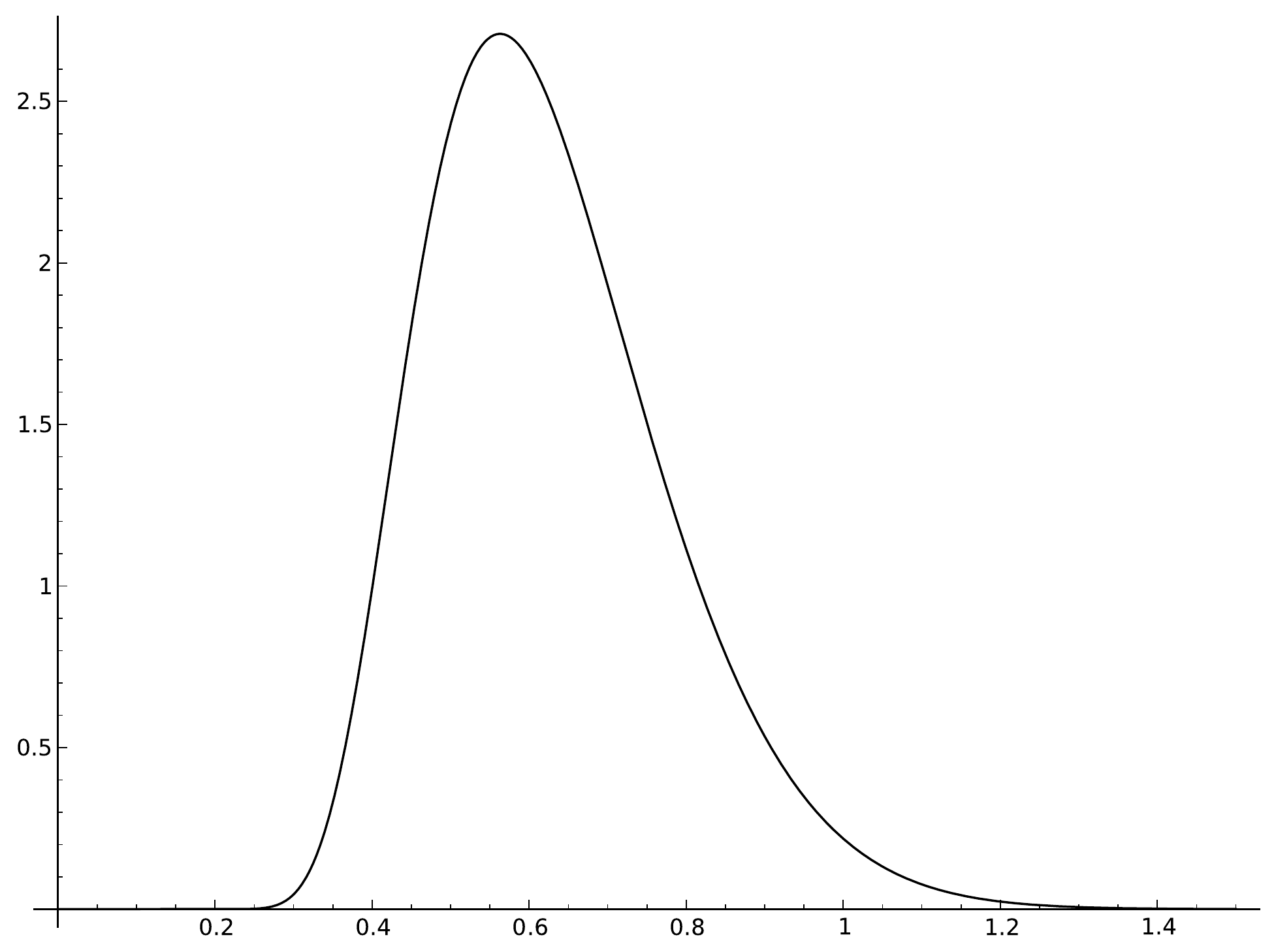}
\caption{The Airy distribution.  A rescaling gives the limiting approximation of the area shown in \ref{fig2}.}
\label{fig3}
\end{center}
\end{figure}

\paragraph{Motivational notes for area}
For a parking function $\pi=(\pi_1,\pi_2,\dots,\pi_n)$, the function $\binom{n+1}{2}-(\pi_1+\pi_2+\dots+\pi_n)$ has been called the area of the parking function. This appears all over the subject. Consider first the inconvenience $I(\pi)$, the number of extra spaces each car is required to travel past their desired space. For example, if $n=5$ the parking function $(1,3,5,3,1)$ results in the parking
\begin{equation*}\begin{array}{ccccc}
1&1&3&3&5\\[-6pt]
-&-&-&-&-
\end{array}\end{equation*}
with $I(\pi)=2$: the last two cars (the second preferring the space 1 and the second preferring 3) both parked one space off. We show
\begin{equation}
I(\pi)=\binom{n+1}{2}-(\pi_1+\dots+\pi_n).
\label{55}
\end{equation}
When $n=5$, for $\pi$ given above,
\begin{equation*}
I(\pi)=\binom{6}{2}-(1+3+5+3+1)=2.
\end{equation*}

\begin{figure}[htb]
\centering
\includegraphics[keepaspectratio, scale=0.3,clip]{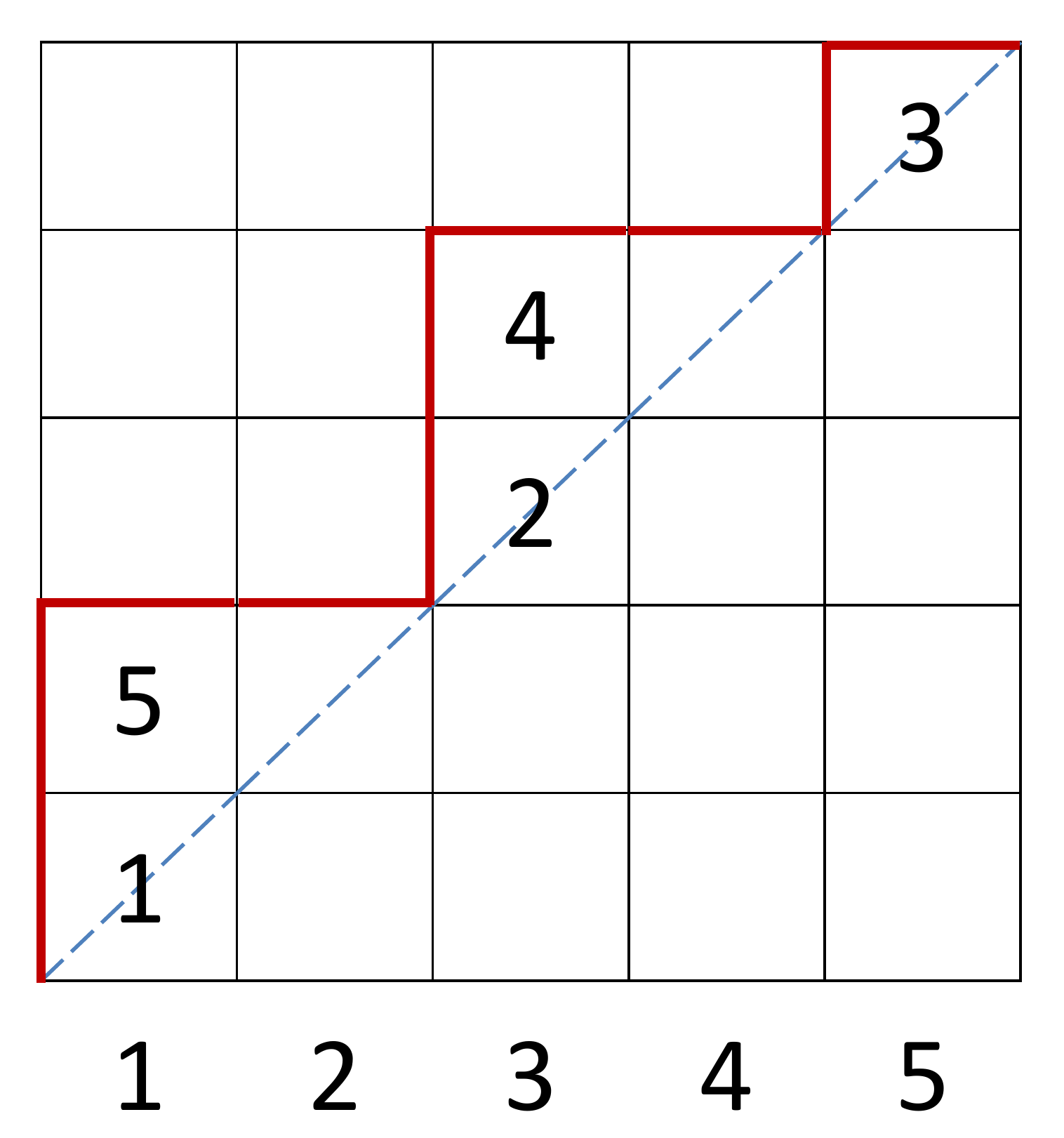}
\caption{Dyck path for $\pi=(1,3,5,3,1)$.}
\label{fig4}
\end{figure}

To motivate the name area --- and to prove \eqref{55} --- it is helpful to have a coding of $\pfn$ as Dyck paths in an $n\times n$ grid, as shown in \ref{fig4}. The construction, as first suggested by Garsia, is, from $\pi$:
\begin{enumerate}
\item Label the boxes at the start of column 1 with the position of the ones in $\pi$.
\item Moving to the next highest label, say $i'$, label the boxes in column $i'$ with the positions of the $i'$'s, starting at height $\#1^s+1$.
\item Continue with $i''$, $i'''$, \textellipsis.
\item Draw a Dyck path, by connecting the left edges of each number in the lattice.
\end{enumerate}

\noindent Clearly $\pi$ can be reconstructed from this data. The condition $\#\{k:\pi_k\leq i\}\geq i$ insures that a Dyck path (having no edge below the main diagonal) results. Define area($\pi$) as the number of labeled boxes strictly above the main diagonal, so $\text{area}(\pi)=2$ in the figure.

\begin{lem}
In a parking function $\pi$,
\begin{equation*}
\textup{area}=\textup{inconvenience distance}=\binom{n+1}2-\sum_{i=1}^n\pi_i.
\end{equation*}
\end{lem}

\begin{proof}
The proof is easiest if $\pi_1\leq\pi_2\leq\dots\leq\pi_n$. In this case the resulting lattice diagram has car $i$ in row $i$ (numbered from the bottom) for all $i$. Then the distance that any driver $i$ drives past his desired space is $i-\pi_i$, since a driver will always park in the next available space, i.e., the $i$th. Moreover, in the lattice diagram, there are $i-1$ complete squares in the $i$th row above the main diagonal, and all but the first $\pi_i-1$ of them are under the Dyck path. Thus complete squares in the $i$th row correspond to spaces the $i$th driver passes and wants to park in.

In the general case, let $\alpha$ give the permutation formed by reading the cars in the lattice diagram from the bottom row to the top row. Let $\beta$ give the permutation formed by the order in which the cars finally park. Similarly, the number of complete squares under the Dyck path in row $i$ is $i-\pi_{\alpha_i}$. The distance any driver drives past his desired space is $\beta_i-\pi_i$. Relying throughout on the fact that we may sum over $\{1,\dots,n\}$ in the order prescribed by any permutation, the area is
\begin{equation*}
\sum_ii-\pi_{\alpha_i}=\binom{n+1}2-\sum_i\pi_i=\sum_i\beta_i-\pi_i,
\end{equation*}
which is exactly the inconvenience distance of the parking function.
\end{proof}
\begin{remark}We can use what we know about $E(\pi_1)$ to determine that average area is $\frac{\sqrt{2\pi}}{4}n^{3/2}$ plus lower order terms. \citet{kung2003exact} computes this and higher moments of the average area of parking functions (and their generalizations.)  Interestingly, \citet[volume 3, third edition, p.733]{knuth1998art} shows $ \frac{\sum\pi_i}{{n\choose k}}$ is the number of connected graphs with $n+k$ edges on $n+1$ labeled vertices. 
\end{remark}
The previous theorems all flow from the following.

\begin{thm}
For $\pi$ uniformly chosen in $\pfn$ and $F^\pi$ defined by \eqref{51},
\begin{equation*}
\sqrt{n}\left[F^\pi(x)-x\right]_{0\leq x\leq 1}\Longrightarrow(E_x)_{0\leq x\leq 1}
\end{equation*}
(weak convergence of processes on $[0,1]$), with $E_x$ the Brownian excursion \citep[p.~75]{ito}.
\label{thm14}
\end{thm}

\begin{proof}
\citet*[Prop.~4.1]{chassaing} give a 1:1 correspondence $\pi\leftrightarrow\mathcal{J}(\pi)$ between parking functions and labeled trees, rooted at 0, with $n+1$ vertices. The bijection satisfies $y_k(\mathcal{J}(\pi))=\hy_k(\pi)$ for $k=0,1,2,\dots,n$. Here for a tree $t$, $y_k(t)$ denotes the length of the queue lengths in a breadth-first search of $t$ and $\hy_k(\pi)=\ha_0(\pi)+\dots+\ha_k(\pi)-k$ with $\ha_k(\pi)=\#\{k:\pi_k=i\}$ (their parking functions start at 0). Thus $\hy_i(\pi)=\#\{i:\pi_i\leq k\}-k$ and
\begin{equation*}
\frac{\hy_k(\pi)}{n}=F^{\pi}\left(\frac{k}{n}\right)-\frac{k}{n}.
\end{equation*}
In \citet[Sect.~4.1]{chassaing} \citeauthor{chassaing} prove that for a uniformly chosen tree,
\begin{equation*}
\left\{\begin{array}{c}y(t)\\\lfloor nx\rfloor\end{array}\bigg/\sqrt{n}\right\}\Longrightarrow\{E_x\},\qquad0\leq x\leq 1.
\end{equation*}
This implies the result.
\end{proof}

Theorems 11, 12, and 13 now follow from the continuity theorem; the distribution of $E_x$ is well known to be the same as
\begin{equation*}
\sqrt{B_1^0(x)^2+B_2^0(x)^2+B_3^0(x)^2},
\end{equation*}
with $B_i^0(x)$ independent Brownian bridges \citep[p.~79]{ito} and so the square root of a scaled Gamma(3) variate. The distribution of $\max_{0\leq x\leq 1}E_x$ is given by \cite{kaigh1978}. The distribution of $\int_0^1E_x\,dx$ has a long history; see \citet{janson} for a detailed development.

\section{Parking functions and representation theory}\label{sec6}

This section sketches out the connection between parking functions and Macdonald polynomials.  It is \textit{not} intended as a detailed history of either set of objects, which would require many pages.  For more details, we suggest \citet{MR3443860} for the polynomials, \citet{haglund} for the parking functions, or for a longer informal history, the afterward by Adriano Garsia in the second author's Ph.\ D.\ thesis \citep{phdthesis}.    

The intersection between parking functions and Macdonald Polynomials is an exciting part of current algebraic combinatorics research, advanced by the work of Bergeron, Loehr, Garsia, Procesi, Haiman, Haglund, Remmel, Carlsson, Mellit, and many others. The story starts with symmetric functions and Maconald's two-parameter family of bases $P_\lambda(x;q,t)$. After several transformations to get to a basis $\tilde{H}_\lambda(x;q,t)$ --- which should be Schur positive, the proof of which involves diagonal harmonics --- the Shuffle Theorem results in
\begin{equation}\label{shuffle}
\text{Char}(\text{DH}_n)=\sum_{\pi\in\pfn}t^{\text{area}(\pi)}q^{\text{dinv}(\pi)}F_{i\text{des}(\pi)}.
\end{equation}
Our parking functions and area are on the right. Motivating and defining the other terms is our goal.

A healthy part of combinatorics is unified by symmetric function theory. Recall that $f\in Q[x_1,\dots,x_n]$ is symmetric if $f(x_1,\dots,x_n)=f(z_{\sigma_1},\dots,x_{\sigma_n})$ for all $\sigma \in \mathfrak{S}_n$. For example, $p_i(x)=\sum_{j=1}^nx_j^i$ is symmetric, and if $\lambda$ is a partition of $N$ (write $\lambda\vdash N$) with $n_i(\lambda)$ parts of size $i$, then
\begin{equation*}
p_\lambda(x)=\prod_i p_i^{n_i(\lambda)}
\end{equation*}
is symmetric. A fundamental theorem says that as $\lambda$ ranges over partitions of $N$, $\{p_\lambda(x)\}_{\lambda\vdash N}$ is a basis for the homogeneous symmetric polynomials of degree $N$.

The other classical bases are: $\{m_\lambda\}$ the monomial symmetric functions; $\{e_\lambda\}$ the elementary symmetric functions; $\{h_\lambda\}$ the homogeneous symmetric functions; and $\{s_\lambda\}$ the Schur functions. The change of bases matrices between these families code up an amazing amount of combinatorics. For example, going from the power sums to the Schur basis,
\begin{equation}
p_\rho=\sum_\lambda \chi_\rho^\lambda s_\lambda,
\label{61}
\end{equation}
with $\chi_\rho^\lambda$ the $\lambda$th irreducible character of the symmetric group at the $\rho$th conjugacy class.

All of this is magnificently told in \citet{MR3443860} and \citet*{stanley99}. Because of \eqref{61} there is an intimate connection between symmetric function theory and the representation theory of the symmetric group $\mathfrak{S}_n$. Let $R^n$ be the space of class functions on $\mathfrak{S}_n$ and $R=\bigoplus_{n=0}^\infty R^n$. Let $\Lambda$ be the ring of all symmetric functions. Then $R$ is isomorphic to $\Lambda$ via the \textit{characteristic map} (also known as the \textit{Frobenius map}). For $f\in R^n$ taking values $f_\rho$ at cycle type $\rho$,
\begin{equation*}
\text{Ch}(f)=\sum_{|\rho|=n}z_\rho^{-1}f_\rho p_\rho\qquad\text{with }z_\rho=\prod_ii^{n_i(\rho)}n_i(\rho)!.
\end{equation*}
The linear extension of Ch to all of $R$ is an isometry for the usual inner product on $\mathfrak{S}_n$ and the Hall inner product $\langle p_\lambda|p_\rho\rangle=\delta_{\lambda\rho}z_\rho$.

Following the classical bases, a host of other bases of $\Lambda$ began to be used. Statisticians used zonal polynomials to perform analysis of covariance matrices of Gaussian samples. Group theorists developed Hall--Littlewood polynomials to describe the subgroup structure of Abelian groups and the irreducible characters of $GL_n(\mathbb{F}_q)$. An amazing unification was found by Ian Macdonald. He introduced a two-parameter family of bases $P_\lambda(x;q,t)$. These can be defined by deforming the Hall inner product to
\begin{equation*}
\langle p_\lambda|p_\mu\rangle=\delta_{\lambda\mu}z_\lambda\prod_i\frac{1-q^{\lambda_i}}{1-t^{\lambda_i}}\qquad(p_\lambda,p_\mu\text{ are power sums}).
\end{equation*}
Then $P_\lambda(x;q,t)=P_\lambda$ is uniquely defined by
\begin{enumerate}
\item orthonormality, i.e., $\langle P_\lambda|P_\mu\rangle=\delta_{\lambda\mu}$;
\item triangularity, i.e., $P_\lambda$ is upper triangular in the $\{m_\mu\}$ basis.
\end{enumerate}

\noindent Remarkably, specializing parameters gave essentially all the other bases: when $t=q$ the Schur functions emerge; when $q=0$ we have the Hall--Littlewoods; setting $q=t^2$ and letting $t\to1$ gives the Jack symmetric functions $J_\alpha^2$; and setting $\alpha=2$ gives zonal polynomials.

Early calculations of $P_\lambda$ showed that it was not polynomial in $q$ and $t$; multiplication by a factor fixed this, giving $J_\lambda$ [Macdonald, Sect.~8.1]. Further calculation showed these were almost Schur positive and, expressed in a $t$-deformation $s_\tau(x;t)$ of the Schur functions, Macdonald's Schur positivity conjecture [Macdonald Sect.~8.18] suggests that $J_\lambda$ expands as a positive integral polynomial in $q$ and $t$.

Adriano Garsia \citep{garsia93} suggested a further modification using a homomorphism of $\Lambda$ (plethysm) to get modified Macdonald polynomials $\tilhlam[x;q,t]$. These restrict to classical bases and have the advantage of being (conjecturally) simply Schur positive.

Early attempts to prove Schur positivity and to find more convenient formulas for $\tilhlam$ were representation-theoretic using the characteristic map: if $\tilhlam$ was Schur positive, it must be the image of a representation of $\mathfrak{S}_n$. The $n!$ conjecture \citep{garsia93}, proved nearly a decade later \citep{haiman01}, states that $\{\tilhlam\}$ is the image of the representation given by $\mathcal{L}[\partial x\partial y\Delta_\lambda]$ the linear span of the derivatives of $\Delta_\lambda$ (span $\partial x^a\partial y^b\Delta_\lambda$) where $\Delta_\lambda$ is a polynomial based on the diagram of $\lambda$. Each (irreducible) representation of the symmetric group occuring in $\mathcal{L}[\partial x\partial y\Delta_\lambda]$ is thus sent by characteristic map to a Schur function; when summed together (along with $t$ and $q$ to give the degree in the $x$ or $y$ variables) they give an expression for $\tilde{H}_\lambda$, and thus Macondald's Schur positivity conjecture. 

\begin{example}
For $\lambda=(4,3,2)$,
\begin{equation*}
\begin{array}{|c|c|c|c|}
\hline
02&12&&\\\hline
01&11&21&\\\hline
00&10&20&30\\\hline
\end{array}\longrightarrow\Delta_{432}=\det(x_i^{p_j}y_i^{q_j})_{1\leq i,j\leq n},
\end{equation*}
where $(p_j,q_j)$ run over the indices in the diagram $[(0,0),\dots,(1,2)]$.
\end{example}

Combinatorial methods for studying the irreducible representations of this space, and thus giving a directly computable definition for the Macdonald polynomials eluded researchers for a number of years. They naturally began studying a larger space which contained $\mathcal{L}[\partial x\partial y\Delta_\lambda]$ for all partitions $\lambda$ of $n$; the resulting space was called the diagonal harmonics and there, the Frobenius image proved easier (at least conjecturally) to compute.
\begin{equation*}
\text{DH}_n=\left\{f\in \mathbb{Q}[\bm{x},\bm{y}]:\sum_{i=1}^n\partial_{x_i}^r\partial_{y_i}^sf(\bm{x},\bm{y})=0\text{ for all }r,s\geq0,\ r+s>0\right\}.
\end{equation*}
This space (or rather an isomorphic space, the diagonal coinvariants) is can also be seen as as the space of polynomials $\mathbb{C}[x_1\dots,x_n,y_1,\dots,y_n]$ moded out by the ideal generated by polynomials invariant under the diagonal action $(x_1\dots,x_n,y_1,\dots,y_n)^\sigma=(x_{\sigma_1},\dots,x_{\sigma_n},y_{\sigma_1},\dots,y_{\sigma_n})$. Of course, the symmetric group acts on this quotient. The characteristic image of this representation was conjecturally identified (the shuffle conjecture) in the 90s \citep{haglund2005combinatorial} and recently proved by \citet{carlsson2015proof} to be describable as a weighted sum of parking functions,
\begin{equation*}
\text{Char}(\text{DH}_n)=\sum_{\pi\in\pfn}t^{\text{area}(\pi)}q^{\text{dinv}(\pi)}F_{i\text{des}(\pi)}.
\end{equation*}
with area($\pi$) as in \ref{sec5}, dinv($\pi$) a second statistic, and $F$ the quasi-symmetric function \citep{stanley99}. This is known to be a positive integral combination of Schur functions. Finally, ides is closely related to the number of weak descents in $\pi$, introduced in \ref{sec4}. (In fact, equivalent formulations of the theorem use the same precise characterization.) The original shuffle conjecture led to a further conjecture of Haglund \citep{haglund2004combinatorial} and proof by Haglund and Haiman \citep{haglund2005combinatorial} of a combinatorial description of the Macdonald polynomials (with statistics related to the parking function statistics, but much more complicated). 

In summary, the study of the area and descent structure of parking functions led directly to the discovery of a direct and combinatorial formula for the Macdonald Polynomials, arguably the most important symmetric function basis. Moreover, facts about area and descent structure can be translated into information about the degrees of polynomials that occur in any irreducible representations of $\mathfrak{S}_n$ occurring in DH$_n$.

\section{Some open problems}\label{sec7}

There is one feature of $f\in\calfn$ that we have not explored for $\pfn$; this is the cycle structure under iteration. For $f\in\calfn$, iteration produces a disjoint union of directed cycles with trees coming into the cycles. There is remarkable and detailed knowledge of the properties of these graphs for uniformly chosen $f$: a typical cycle has length about $\sqrt{n}$ as does a typical tree. Maxima and joint distributions are also carefully studied. See \citet{kolchin1986random} or \citet{harris1960probability} for the classical theory of random mappings. Jenny Hansen's many contributions here are notable; see \citet{hansen}. The paper by  \citet{ald-pit} relates natural features of a mapping to natural features of the Brownian bridge. It is understandable to expect a parallel coupling between $\pfn$ and Brownian excursion but this has not been worked out.

There are three other developments connected to parking functions where the program outlined in our paper can be attempted. The first is so-called rational parking functions or $(m,n)$ parking functions. See \citet{gorsky}, \citet{hikita2014affine}, \citet{bergeron2015compositional} and their many references. These are at the forefront of current research with applications to things like the cohomology of affine springer fibers. There are $m^{n-1}$ of these things: what does a typical one look like?

Our parking functions have a close connection to the symmetric group. There are interesting natural analogs for other reflection groups. For definitions and basic properties, see \citet{MR3281144}. We have not seen \textit{any} probabilistic development in this direction, although certain special cases overlap with \citet{kung2003expected}.

The third development is to $G$-parking functions \citep{postnikov}. Here $G$ is a general graph. There is a general definition which specializes to $\pfn$ for the complete graph $K_n$. These $G$-parking functions have amazing connections to algebraic geometry, Riemann surface theory, and much else through their connection to chip firing games and sand pile models \citep{lopez1997chip}. One way to get started in this area is to ask ``What does a random $G$-parking function look like?''

Of course, our probabilistic path is just one thread in an extremely rich fabric. We have found it useful in getting started and hope that our readers will, too.
\begin{ack} The authors would like to thank Alexei Borodin, Sourav Chatterjee, Susan Holmes, Joseph Kung, Jim Pitman, Arun Ram, Kannan Soundararajan, Richard Stanley and Catherine Yan for their insightful comments on the project.  Both author's interest in this project and in parking functions is in large part a tribute to Adriano Garsia's infectious enthusiasm for the subject; we'd like to thank him for his influence as well as his particular comments on this paper. 
The first author's research was partially supported by NSF grant DMS-1208775; the second's by NSF grant DMS 1303761.\end{ack}

\bibliography{parking}

\def\cprime{$'$}
\begin{thebibliography}{56}
\expandafter\ifx\csname natexlab\endcsname\relax\def\natexlab#1{#1}\fi
\expandafter\ifx\csname url\endcsname\relax
  \def\url#1{\texttt{#1}}\fi
\expandafter\ifx\csname urlprefix\endcsname\relax\def\urlprefix{URL }\fi
\providecommand{\eprint}[2][]{\url{#2}}

\bibitem[{Aldous and Pitman(1994)}]{ald-pit}
\textsc{Aldous, D.~J.} and \textsc{Pitman, J.} (1994).
\newblock Brownian bridge asymptotics for random mappings.
\newblock \textit{Random Structures Algorithms}, \textbf{5} 487--512.
\newblock \urlprefix\url{dx.doi.org/10.1002/rsa.3240050402}.

\bibitem[{Armstrong et~al.(2015)Armstrong, Reiner and Rhoades}]{MR3281144}
\textsc{Armstrong, D.}, \textsc{Reiner, V.} and \textsc{Rhoades, B.} (2015).
\newblock Parking spaces.
\newblock \textit{Adv. Math.}, \textbf{269} 647--706.
\newblock \urlprefix\url{http://dx.doi.org/10.1016/j.aim.2014.10.012}.

\bibitem[{Baik et~al.(1999)Baik, Deift and Johansson}]{baik1999distribution}
\textsc{Baik, J.}, \textsc{Deift, P.} and \textsc{Johansson, K.} (1999).
\newblock On the distribution of the length of the longest increasing
  subsequence of random permutations.
\newblock \textit{Journal of the American Mathematical Society}, \textbf{12}
  1119--1178.

\bibitem[{Beck et~al.(2015)Beck, Berrizbeitia, Dairyko, Rodriguez, Ruiz and
  Veeneman}]{MR3383893}
\textsc{Beck, M.}, \textsc{Berrizbeitia, A.}, \textsc{Dairyko, M.},
  \textsc{Rodriguez, C.}, \textsc{Ruiz, A.} and \textsc{Veeneman, S.} (2015).
\newblock Parking functions, {S}hi arrangements, and mixed graphs.
\newblock \textit{Amer. Math. Monthly}, \textbf{122} 660--673.
\newblock
  \urlprefix\url{http://dx.doi.org/10.4169/amer.math.monthly.122.7.660}.

\bibitem[{Bergeron et~al.(2015)Bergeron, Garsia, Leven and
  Xin}]{bergeron2015compositional}
\textsc{Bergeron, F.}, \textsc{Garsia, A.}, \textsc{Leven, E.~S.} and
  \textsc{Xin, G.} (2015).
\newblock Compositional (km, kn)-shuffle conjectures.
\newblock \textit{International Mathematics Research Notices} rnv272.

\bibitem[{Billingsley(2012)}]{billingsley2012probability}
\textsc{Billingsley, P.} (2012).
\newblock \textit{Probability and Measure}.
\newblock Wiley Series in Probability and Statistics, Wiley.
\newblock \urlprefix\url{https://books.google.com/books?id=a3gavZbxyJcC}.

\bibitem[{Bobkov(2005)}]{bobkov2005generalized}
\textsc{Bobkov, S.~G.} (2005).
\newblock Generalized symmetric polynomials and an approximate de finetti
  representation.
\newblock \textit{Journal of Theoretical Probability}, \textbf{18} 399--412.

\bibitem[{Borodin et~al.(2010)Borodin, Diaconis and Fulman}]{pd-borodin}
\textsc{Borodin, A.}, \textsc{Diaconis, P.} and \textsc{Fulman, J.} (2010).
\newblock On adding a list of numbers (and other one-dependent determinantal
  processes).
\newblock \textit{Bull. Amer. Math. Soc. (N.S.)}, \textbf{47} 639--670.
\newblock \urlprefix\url{dx.doi.org/10.1090/S0273-0979-2010-01306-9}.

\bibitem[{Briggs et~al.(2009)Briggs, Song and Prellberg}]{briggs2009note}
\textsc{Briggs, K.}, \textsc{Song, L.} and \textsc{Prellberg, T.} (2009).
\newblock A note on the distribution of the maximum of a set of poisson random
  variables.
\newblock \textit{arXiv preprint arXiv:0903.4373}.

\bibitem[{Carlsson and Mellit(2015)}]{carlsson2015proof}
\textsc{Carlsson, E.} and \textsc{Mellit, A.} (2015).
\newblock A proof of the shuffle conjecture.
\newblock \textit{arXiv preprint arXiv:1508.06239}.

\bibitem[{Chassaing and Marckert(2001)}]{chassaing}
\textsc{Chassaing, P.} and \textsc{Marckert, J.-F.} (2001).
\newblock Parking functions, empirical processes, and the width of rooted
  labeled trees.
\newblock \textit{Electron. J. Combin.}, \textbf{8} Research Paper 14, 19 pp.
  (electronic).
\newblock \urlprefix\url{combinatorics.org/Volume_8/Abstracts/v8i1r14.html}.

\bibitem[{Chebikin and Postnikov(2010)}]{MR2576844}
\textsc{Chebikin, D.} and \textsc{Postnikov, A.} (2010).
\newblock Generalized parking functions, descent numbers, and chain polytopes
  of ribbon posets.
\newblock \textit{Adv. in Appl. Math.}, \textbf{44} 145--154.
\newblock \urlprefix\url{http://dx.doi.org/10.1016/j.aam.2009.02.004}.

\bibitem[{Cori and Rossin(2000)}]{MR1756151}
\textsc{Cori, R.} and \textsc{Rossin, D.} (2000).
\newblock On the sandpile group of dual graphs.
\newblock \textit{European J. Combin.}, \textbf{21} 447--459.
\newblock \urlprefix\url{http://dx.doi.org/10.1006/eujc.1999.0366}.

\bibitem[{Diaconis and Freedman(1980)}]{diaconis1980finite}
\textsc{Diaconis, P.} and \textsc{Freedman, D.} (1980).
\newblock Finite exchangeable sequences.
\newblock \textit{The Annals of Probability} 745--764.

\bibitem[{Edwards(1974)}]{edwards}
\textsc{Edwards, H.~M.} (1974).
\newblock \textit{Riemann's zeta function}.
\newblock Pure and Applied Mathematics, Vol.\ 58, Academic Press Publishers,
  New York-London.

\bibitem[{Eu et~al.(2005)Eu, Fu and Lai}]{eu2005enumeration}
\textsc{Eu, S.-P.}, \textsc{Fu, T.-S.} and \textsc{Lai, C.-J.} (2005).
\newblock On the enumeration of parking functions by leading terms.
\newblock \textit{Advances in Applied Mathematics}, \textbf{35} 392--406.

\bibitem[{Feller(1971)}]{feller}
\textsc{Feller, W.} (1971).
\newblock \textit{{An Introduction to Probability Theory and its Applications.
  Vol. II}}.
\newblock 3rd ed. John Wiley \& Sons Inc., New York.

\bibitem[{Flajolet et~al.(1998)Flajolet, Poblete and
  Viola}]{flajolet1998analysis}
\textsc{Flajolet, P.}, \textsc{Poblete, P.} and \textsc{Viola, A.} (1998).
\newblock On the analysis of linear probing hashing.
\newblock \textit{Algorithmica}, \textbf{22} 490--515.

\bibitem[{Foata and Riordan(1974)}]{foata1974mappings}
\textsc{Foata, D.} and \textsc{Riordan, J.} (1974).
\newblock Mappings of acyclic and parking functions.
\newblock \textit{Aequationes Mathematicae}, \textbf{10} 10--22.

\bibitem[{Garsia and Haiman(1993)}]{garsia93}
\textsc{Garsia, A.~M.} and \textsc{Haiman, M.} (1993).
\newblock {A graded representation model for Macdonald's polynomials}.
\newblock \textit{Proc. Nat. Acad. Sci. U.S.A.}, \textbf{90} 3607--3610.
\newblock \urlprefix\url{dx.doi.org/10.1073/pnas.90.8.3607}.

\bibitem[{Gessel and Seo(2006)}]{gessel2006refinement}
\textsc{Gessel, I.~M.} and \textsc{Seo, S.} (2006).
\newblock A refinement of cayley’s formula for trees.
\newblock \textit{Electron. J. Combin}, \textbf{11} R27.

\bibitem[{Gorsky et~al.(2016)Gorsky, Mazin and Vazirani}]{gorsky}
\textsc{Gorsky, E.}, \textsc{Mazin, M.} and \textsc{Vazirani, M.} (2016).
\newblock Affine permutations and rational slope parking functions.
\newblock February (online ahead of print),
  \urlprefix\url{dx.doi.org/10.1090/tran/6584}.

\bibitem[{Haglund(2004)}]{haglund2004combinatorial}
\textsc{Haglund, J.} (2004).
\newblock A combinatorial model for the macdonald polynomials.
\newblock \textit{Proceedings of the National Academy of Sciences of the United
  States of America}, \textbf{101} 16127--16131.

\bibitem[{Haglund(2008)}]{haglund}
\textsc{Haglund, J.} (2008).
\newblock \textit{{The {$q$},{$t$}-Catalan numbers and the space of diagonal
  harmonics}}, vol.~41 of \textit{University Lecture Series}.
\newblock American Mathematical Society, Providence, RI.

\bibitem[{Haglund et~al.(2005)Haglund, Haiman, Loehr, Remmel, Ulyanov
  et~al.}]{haglund2005combinatorial}
\textsc{Haglund, J.}, \textsc{Haiman, M.}, \textsc{Loehr, N.}, \textsc{Remmel,
  J.}, \textsc{Ulyanov, A.} \textsc{et~al.} (2005).
\newblock A combinatorial formula for the character of the diagonal
  coinvariants.
\newblock \textit{Duke Mathematical Journal}, \textbf{126} 195--232.

\bibitem[{Haiman(2001)}]{haiman01}
\textsc{Haiman, M.} (2001).
\newblock {Hilbert schemes, polygraphs and the Macdonald positivity
  conjecture}.
\newblock \textit{J. Amer. Math. Soc.}, \textbf{14} 941--1006 (electronic).
\newblock \urlprefix\url{dx.doi.org/10.1090/S0894-0347-01-00373-3}.

\bibitem[{Haiman(2002)}]{haiman02}
\textsc{Haiman, M.} (2002).
\newblock {Vanishing theorems and character formulas for the Hilbert scheme of
  points in the plane}.
\newblock \textit{Invent. Math.}, \textbf{149} 371--407.
\newblock \urlprefix\url{dx.doi.org/10.1007/s002220200219}.

\bibitem[{Hansen(1989)}]{hansen}
\textsc{Hansen, J.~C.} (1989).
\newblock A functional central limit theorem for random mappings.
\newblock \textit{Ann. Probab.}, \textbf{17} 317--332.
\newblock \urlprefix\url{jstor.org/stable/2244213}.

\bibitem[{Harris(1960)}]{harris1960probability}
\textsc{Harris, B.} (1960).
\newblock Probability distributions related to random mappings.
\newblock \textit{The Annals of Mathematical Statistics} 1045--1062.

\bibitem[{Hicks(2013)}]{phdthesis}
\textsc{Hicks, A.} (2013).
\newblock \textit{Parking Function Polynomials and Their Relation to the
  Shuffle Conjecture}.
\newblock Ph.D. thesis, UCSD, The address of the publisher.
\newblock With historical note by Adriano Garsia.
  \url{https://escholarship.org/uc/item/8tp1q52k#}.

\bibitem[{Hikita(2014)}]{hikita2014affine}
\textsc{Hikita, T.} (2014).
\newblock Affine springer fibers of type a and combinatorics of diagonal
  coinvariants.
\newblock \textit{Advances in Mathematics}, \textbf{263} 88--122.

\bibitem[{It{\^o} and McKean(1965)}]{ito}
\textsc{It{\^o}, K.} and \textsc{McKean, H.~P., Jr.} (1965).
\newblock \textit{{Diffusion Processes and Their Sample Paths}}.
\newblock Academic Press, Inc., New York.

\bibitem[{Janson(2013)}]{janson}
\textsc{Janson, S.} (2013).
\newblock {Moments of the location of the maximum of Brownian motion with
  parabolic drift}.
\newblock \textit{Electron. Commun. Probab.}, \textbf{18} 8.
\newblock \urlprefix\url{dx.doi.org/10.1214/ECP.v18-2330}.

\bibitem[{Kaigh(1978)}]{kaigh1978}
\textsc{Kaigh, W.} (1978).
\newblock An elementary derivation of the distribution of the maxima of
  brownian meander and brownian excursion.
\newblock \textit{Rocky Mountain J. Math.}, \textbf{8} 641--646.
\newblock \urlprefix\url{http://dx.doi.org/10.1216/RMJ-1978-8-4-641}.

\bibitem[{Knuth(1998{\natexlab{a}})}]{knuth1998art}
\textsc{Knuth, D.~E.} (1998{\natexlab{a}}).
\newblock \textit{The art of computer programming: sorting and searching},
  vol.~3.
\newblock Pearson Education.

\bibitem[{Knuth(1998{\natexlab{b}})}]{knuth1998linear}
\textsc{Knuth, D.~E.} (1998{\natexlab{b}}).
\newblock Linear probing and graphs.
\newblock \textit{Algorithmica}, \textbf{22} 561--568.

\bibitem[{Kolchin(1986)}]{kolchin1986random}
\textsc{Kolchin, V.~F.} (1986).
\newblock Random mappings. translated from the {R}ussian. with a foreword by
  {SRS} {V}aradhan. {T}ranslation {S}eries in {M}athematics and {E}ngineering.
\newblock \textit{Optimization Software, Inc., Publications Division, New
  York}.

\bibitem[{Kolchin et~al.(1978)Kolchin, Sevast{\cprime}yanov and
  Chistyakov}]{kolchin1978random}
\textsc{Kolchin, V.~F.}, \textsc{Sevast{\cprime}yanov, B.~A.} and
  \textsc{Chistyakov, V.~P.} (1978).
\newblock \textit{Random allocations}.
\newblock Vh Winston.

\bibitem[{Konheim and Weiss(1966)}]{konheim}
\textsc{Konheim, A.~G.} and \textsc{Weiss, B.} (1966).
\newblock An occupancy discipline and applications.
\newblock \textit{SIAM J. Appl. Math.}, \textbf{14} 1266--1274.
\newblock \urlprefix\url{dx.doi.org/10.1137/0114101}.

\bibitem[{Kung and Yan(2003{\natexlab{a}})}]{kung2003exact}
\textsc{Kung, J.~P.} and \textsc{Yan, C.} (2003{\natexlab{a}}).
\newblock Exact formulas for moments of sums of classical parking functions.
\newblock \textit{Advances in Applied Mathematics}, \textbf{31} 215--241.

\bibitem[{Kung and Yan(2003{\natexlab{b}})}]{kung2003expected}
\textsc{Kung, J.~P.} and \textsc{Yan, C.} (2003{\natexlab{b}}).
\newblock Expected sums of general parking functions.
\newblock \textit{Annals of Combinatorics}, \textbf{7} 481--493.

\bibitem[{L{\'o}pez(1997)}]{lopez1997chip}
\textsc{L{\'o}pez, C.~M.} (1997).
\newblock Chip firing and the tutte polynomial.
\newblock \textit{Annals of Combinatorics}, \textbf{1} 253--259.

\bibitem[{Macdonald(2015)}]{MR3443860}
\textsc{Macdonald, I.~G.} (2015).
\newblock \textit{Symmetric functions and {H}all polynomials}.
\newblock 2nd ed. Oxford Classic Texts in the Physical Sciences, The Clarendon
  Press, Oxford University Press, New York.

\bibitem[{Majumdar and Comtet(2005)}]{MR2151223}
\textsc{Majumdar, S.~N.} and \textsc{Comtet, A.} (2005).
\newblock Airy distribution function: from the area under a {B}rownian
  excursion to the maximal height of fluctuating interfaces.
\newblock \textit{J. Stat. Phys.}, \textbf{119} 777--826.
\newblock \urlprefix\url{http://dx.doi.org/10.1007/s10955-005-3022-4}.

\bibitem[{Pitman(2002)}]{pitman}
\textsc{Pitman, J.} (2002).
\newblock {Forest volume decompositions and Abel--Cayley--Hurwitz multinomial
  expansions}.
\newblock \textit{J. Combin. Theory Ser. A}, \textbf{98} 175--191.
\newblock \urlprefix\url{dx.doi.org/10.1006/jcta.2001.3238}.

\bibitem[{Postnikov and Shapiro(2004)}]{postnikov}
\textsc{Postnikov, A.} and \textsc{Shapiro, B.} (2004).
\newblock Trees, parking functions, syzygies, and deformations of monomial
  ideals.
\newblock \textit{Trans. Amer. Math. Soc.}, \textbf{356} 3109--3142
  (electronic).
\newblock \urlprefix\url{dx.doi.org/10.1090/S0002-9947-04-03547-0}.

\bibitem[{Riordan(1968)}]{riordan}
\textsc{Riordan, J.} (1968).
\newblock \textit{{Combinatorial Identities}}.
\newblock John Wiley \& Sons, Inc., New York-London-Sydney.

\bibitem[{Shi(1986)}]{MR835214}
\textsc{Shi, J.~Y.} (1986).
\newblock \textit{The {K}azhdan-{L}usztig cells in certain affine {W}eyl
  groups}, vol. 1179 of \textit{Lecture Notes in Mathematics}.
\newblock Springer-Verlag, Berlin.
\newblock \urlprefix\url{http://dx.doi.org/10.1007/BFb0074968}.

\bibitem[{Smith and Diaconis(1988)}]{smith1988honest}
\textsc{Smith, L.} and \textsc{Diaconis, P.} (1988).
\newblock Honest bernoulli excursions.
\newblock \textit{Journal of applied probability} 464--477.

\bibitem[{Stanley(1997{\natexlab{a}})}]{stanley1997enumerative}
\textsc{Stanley, R.~P.} (1997{\natexlab{a}}).
\newblock \textit{Enumerative Combinatorics. Vol. 1, vol. 49 of Cambridge
  Studies in Advanced Mathematics}.
\newblock Cambridge University Press, Cambridge.

\bibitem[{Stanley(1997{\natexlab{b}})}]{MR1444167}
\textsc{Stanley, R.~P.} (1997{\natexlab{b}}).
\newblock Parking functions and noncrossing partitions.
\newblock \textit{Electron. J. Combin.}, \textbf{4} Research Paper 20, approx.
  14 pp. (electronic).
\newblock The Wilf Festschrift (Philadelphia, PA, 1996),
  \urlprefix\url{http://www.combinatorics.org/Volume_4/Abstracts/v4i2r20.html}.

\bibitem[{Stanley(1998)}]{MR1627378}
\textsc{Stanley, R.~P.} (1998).
\newblock Hyperplane arrangements, parking functions and tree inversions.
\newblock In \textit{Mathematical essays in honor of {G}ian-{C}arlo {R}ota
  ({C}ambridge, {MA}, 1996)}, vol. 161 of \textit{Progr. Math.} Birkh\"auser
  Boston, Boston, MA, 359--375.

\bibitem[{Stanley(1999{\natexlab{a}})}]{stanley99}
\textsc{Stanley, R.~P.} (1999{\natexlab{a}}).
\newblock \textit{{Enumerative Combinatorics. Vol.\ 2}}, vol.~62 of
  \textit{Cambridge Studies in Advanced Mathematics}.
\newblock Cambridge University Press, Cambridge.
\newblock With a foreword by Gian-Carlo Rota and Appendix 1 by Sergey Fomin.

\bibitem[{Stanley(1999{\natexlab{b}})}]{stanleyenumerative}
\textsc{Stanley, R.~P.} (1999{\natexlab{b}}).
\newblock \textit{Enumerative Combinatorics, vol. II.}
\newblock Cambridge University Press, Cambridge.

\bibitem[{Stanley and Pitman(2002)}]{MR1902680}
\textsc{Stanley, R.~P.} and \textsc{Pitman, J.} (2002).
\newblock A polytope related to empirical distributions, plane trees, parking
  functions, and the associahedron.
\newblock \textit{Discrete Comput. Geom.}, \textbf{27} 603--634.
\newblock \urlprefix\url{http://dx.doi.org/10.1007/s00454-002-2776-6}.

\bibitem[{Yan(2015)}]{yan}
\textsc{Yan, C.~H.} (2015).
\newblock Parking functions.
\newblock In \textit{{Handbook of Enumerative Combinatorics}}. Discrete Math.
  Appl., CRC Press, Boca Raton, FL, 835--893.

\end{thebibliography}

\end{document}